\documentclass[reqno,tbtags,intlimits,a4paper,oneside,11pt]{amsart}

\calclayout

\usepackage{amsfonts,amsmath,amssymb}
\usepackage{amsfonts,amsmath,amssymb,float}
\usepackage{tikz}
\usetikzlibrary{decorations.markings}
\usepackage{mathscinet,natbib,hyperref}
\bibpunct{[}{]}{,}{n}{,}{;}

\newtheorem{Theorem}{Theorem}[section]
\newtheorem{proposition}[Theorem]{Proposition}
\newtheorem{lemma}[Theorem]{Lemma}
\newtheorem{remark}[Theorem]{Remark}

\begin{document}

\title{Automorphic Lie algebras and modular forms}

\author{V. Knibbeler}\address{Department of Mathematical Sciences,
Loughborough University, Loughborough LE11 3TU, UK}
\email{V.Knibbeler@lboro.ac.uk}

\author{S. Lombardo}\address{Department of Mathematical Sciences,
Loughborough University, Loughborough LE11 3TU, UK}
\email{S.Lombardo@lboro.ac.uk}

\author{A.P. Veselov}
\address{Department of Mathematical Sciences,
Loughborough University, Loughborough LE11 3TU, UK and
\newline 
Faculty of Mechanics and Mathematics, Moscow State University and
Steklov Mathematical Institute, Moscow, Russia}
\email{A.P.Veselov@lboro.ac.uk}

\newcommand{\mc}[1]{\mathcal{#1}}
\newcommand{\mb}[1]{\mathbb{#1}}
\newcommand{\mf}[1]{\mathfrak{#1}}

\newcommand{\modaut}[2][(\tau)]{\Phi_{#2}#1}

\def\beq#1#2\eeq{%
        \begin{equation}%
        \label{#1}%
            #2%
        \end{equation}%
    }
    
\newcommand{\ba}{\bar{a}}
\newcommand{\backspace}{\!\!\!} 
\newcommand{\n}{n_4}
\newcommand{\nn}{n_6}
\newcommand{\om}{\omega_4^2}
\newcommand{\omm}{\omega_6^2}    

\newcommand{\rd}{\textrm{d}}

\newcommand{\rg}{\Gamma}
\newcommand{\zn}[1]{\mb{Z}/#1\mb{Z}}

\newcommand{\SLNZ}[1][2]{\mathrm{SL}({#1},\mb{Z})}
\newcommand{\SLNC}[1][n]{\mathrm{SL}({#1},\mb{C})}
\newcommand{\slnc}[1][n]{\mf{sl}({#1},\mb{C})}

\newcommand{\mg}{{\Gamma(1)}}
\newcommand{\uh}{\mathbb H^2}
\newcommand{\mfz}[1][\rg]{M^!_\ast(#1)}
\newcommand{\mfk}[2][\rg]{M^!_{#2}(#1)}
\newcommand{\mfpz}[1][\rg]{M_\ast(#1)}
\newcommand{\mfpk}[2][\rg]{M_{#2}(#1)}

\newcommand{\dk}[1]{\rd\tau^{\frac{#1}{2}}}
\newcommand{\dmk}[1]{\rd\tau^{-\frac{#1}{2}}}
\newcommand{\hol}{\mc{F}}

\newcommand{\aliaz}[1][\rg]{\mf{M}^!_\ast(\mf{g},{#1},\rho)}
\newcommand{\alia}[1][\rg]{\mf{M}^!_0(\mf{g},{#1},\rho)}
\newcommand{\aliak}[2][\rg]{\mf{M}^!_{#2}(\mf{g},{#1},\rho)}
\newcommand{\aliavar}[4]{\mf{M}^!_{#1}(#2,#3,#4)}
\newcommand{\aliapz}[1][\rg]{\mf{M}_\ast(\mf{g},{#1},\rho)}
\newcommand{\aliap}[1][\rg]{\mf{M}_0(\mf{g},{#1},\rho)}
\newcommand{\aliapk}[2][\rg]{\mf{M}_{#2}(\mf{g},{#1},\rho)}
\newcommand{\aliapvar}[4]{\mf{M}_{#1}(#2,#3,#4)}

\newcommand{\vvmfz}[1][\rho]{M^!_\ast(#1)}
\newcommand{\vvmf}[1][\rho]{M^!_0(#1)}
\newcommand{\vvmfk}[2][\rho]{M^!_{#2} ( {#1} )}
\newcommand{\vvmfpz}[1][\rho]{M_\ast( {#1} )}
\newcommand{\vvmfp}[1][\rho]{M_0(#1)}
\newcommand{\vvmfpk}[2][\rho]{M_{#2} ( {#1} )}

\newcommand{\roots}{\mc{R}}
\newcommand{\sroots}{\Delta}

\newcommand{\Sym}{\textrm{Sym}}
\newcommand{\Id}{\mathrm{Id}}
\newcommand{\id}{\mathrm{id}}
\newcommand{\tr}{\mathrm{tr}\,}
\newcommand{\im}{\mathrm{Im }} 
\newcommand{\coker}{\textrm{coker}\,} 
\newcommand{\GL}{\mathrm{GL}}
\newcommand{\SL}{\mathrm{SL}}
\newcommand{\SO}{\mathrm{SO}}
\newcommand{\NSO}{\mathrm{O}}
\newcommand{\Sp}{\mathrm{Sp}}
\newcommand{\PSL}{\mathrm{PSL}}
\newcommand{\PSO}{\mathrm{PSO}}
\newcommand{\PGL}{\mathrm{PGL}}
\newcommand{\PG}{\mathrm{P}\mc{G}}
\newcommand{\Hom}{\mathrm{Hom}}
\newcommand{\End}{\mathrm{End}}
\newcommand{\Aut}[1]{\mathrm{Aut}\!\left(#1 \right)}
\newcommand{\Int}[1]{\mathrm{Int}\!\left(#1 \right)}
\newcommand{\Der}[1]{\mathrm{Der}\!\left(#1 \right)}
\newcommand{\Ad}{\mathrm{Ad}}
\newcommand{\ad}{\mathrm{ad}}
\newcommand{\Irr}{\mathrm{Irr}}
\newcommand{\lcm}{\mathrm{\,lcm}}
\newcommand{\codim}{\mathrm{codim}\,}
\newcommand{\diag}{\mathrm{diag}}
\newcommand{\rank}{\mathrm{rank}\,}
\newcommand{\res}{\mathrm{res}\,}
\newcommand{\On}{\mf{O}}

\newcommand\colorchain{gray}
\newcommand\colorp{red}
\newcommand\colorm{blue}
\tikzstyle{root}=[scale=0.5,shape=circle ]
\tikzstyle{2ndroot}=[scale=0.8,shape=circle ]
\tikzstyle{cross}=[scale=0.15,shape=circle, fill]
\tikzstyle{chain}=[color=\colorchain,->,shorten >=8pt,shorten <=8pt, >=stealth]
\tikzstyle{schain}=[color=\colorchain,-,shorten >=4pt,shorten <=4pt]
\tikzstyle{cochainp}=[thick, color=\colorp,->,shorten >=8pt,shorten <=8pt, >=stealth]
\tikzstyle{cochainm}=[thick, color=\colorm,->,shorten >=8pt,shorten <=8pt, >=stealth]
\tikzstyle{scochainp}=[color=\colorp,-,shorten >=3pt,shorten <=3pt]
\tikzstyle{scochainm}=[color=\colorm, dashed,-,shorten >=3pt,shorten <=3pt]
\newcommand{\scaleAA}{0.26}
\newcommand{\scaleBB}{1.0}
\newcommand{\scaleGG}{1.3}
\newcommand{\dangle}{10}
\newcommand{\bend}{15}

\newcommand{\dynkinscale}{0.9}
\newcommand{\dynkintablescale}{0.9}
\newcommand{\dynkinfont}{\footnotesize}
\tikzstyle{dynkinnode}=[draw, color=blue, shape=circle,minimum size=3.5 pt,inner sep=0]
\tikzstyle{ldynkinnode}=[draw, color=blue, shape=circle,minimum size=3.5 pt,inner sep=0]
\tikzstyle{sdynkinnode}=[draw, color=red, shape=circle,minimum size=3.5 pt,inner sep=0]
\tikzstyle{mdynkinnode}=[draw, color=magenta, shape=circle,minimum size=3.5 pt,inner sep=0]
\tikzstyle{fdynkinnode}=[draw, color=blue, shape=circle,minimum size=3.5 pt,inner sep=0,fill=black]
\tikzstyle{fldynkinnode}=[draw, color=blue, shape=circle,minimum size=3.5 pt,inner sep=0,fill=black]
\tikzstyle{fsdynkinnode}=[draw, color=red, shape=circle,minimum size=3.5 pt,inner sep=0,fill=black]
\tikzstyle{fmdynkinnode}=[draw, color=magenta, shape=circle,minimum size=3.5 pt,inner sep=0,fill=black]
\tikzstyle{brace}=[ decorate, decoration={brace, amplitude=5pt}]
\tikzstyle{mbrace}=[decorate, decoration={brace, amplitude=5pt, mirror}]

\begin{abstract}
We introduce and study certain hyperbolic versions of automorphic Lie algebras related to the modular group. 
Let $\rg$ be a finite index subgroup of $\SLNZ[2]$ with an action on a complex simple Lie algebra $\mathfrak g$, which can be extended to $\SLNC[2]$. We show that the Lie algebra of the corresponding $\mf{g}$-valued modular forms is isomorphic to the extension of $\mf{g}$ over the usual modular forms. This establishes a modular analogue of a well-known result by Kac on twisted loop algebras. The case of principal congruence subgroups $\Gamma(N), \, N\leq 6$ are considered in more details in relation to the classical results of Klein and Fricke and the celebrated Markov Diophantine equation. We finish with a brief discussion of the extensions and representations of these Lie algebras.
\end{abstract}

\maketitle

\section{Introduction}
Let $\mathfrak g$ be a simple finite-dimensional complex Lie algebra and $\mathfrak g \otimes_{\mathbb C} \mathbb C[z,z^{-1}]$ be its 
{\it loop algebra}
consisting of Laurent polynomials of $z$ with values in $\mathfrak g.$ It can be considered geometrically as the Lie algebra of Laurent polynomial maps
$$
f: \mathbb C^* \to \mathfrak g,
$$
where $\mathbb C^*=\mathbb C \setminus 0$, which can be considered as a complex version of $S^1.$ 

These infinite-dimensional Lie algebras play a very important role in the theory of $\mathbb Z$-graded Lie algebras initiated in 1968 by Victor Kac \cite{kac1968simple}.
They have a famous unique central extension which together with the derivation $z\frac{\mathrm d}{\mathrm d z}$ form the important class of the affine Lie algebras, but for us the starting point will be the loop algebra itself. It has the following natural twisted versions.

Let $\sigma$ be an automorphism of $\mathfrak g$ with $\sigma^m=\id$ and consider the twisted loop algebra $\mc{L}(\mf{g},\sigma,m)$ consisting of equivariant maps
$f: \mathbb C^* \to \mathfrak g,$ meaning
$$
f(\varepsilon z)=\sigma f(z), \quad \varepsilon^m=1.
$$ 
Kac proved that for every inner automorphism $\sigma$ the corresponding twisted Lie algebra is isomorphic to the untwisted one
as $\mathbb Z$-graded Lie algebras:
\beq{kac}
\mc{L}(\mf{g},\sigma,m) \cong \mc{L}(\mf{g},\id,1). 
\eeq
This important result implies that the isomorphism class of $\mc{L}(\mf{g},\sigma,m)$ only depends on the connected component of the automorphism group of $\mf{g}$ containing $\sigma$ \cite[Theorem 1]{kac1969automorphisms}, \cite[Proposition 8.5]{kac1990infinite}.

The isomorphism (\ref{kac}) can be rewritten in the suggestive form 
\[\left(\mathfrak g \otimes_{\mathbb C} \mathbb C[z,z^{-1}]\right)^{\zn{m}}\cong \mathfrak g \otimes_{\mathbb C} \mathbb C[z,z^{-1}]^{\zn{m}}\] where $\zn{m}$ acts on $\mathbb C[z,z^{-1}]$ by $z\mapsto\varepsilon^{-1} z $. This isomorphism is constructed for every simple Lie algebra $\mf{g}$ and any inner automorphism $\sigma$ of $\mf{g}$ in \cite[Proposition 8.5]{kac1990infinite}. Since $\mathbb C[z,z^{-1}]^{\zn{m}}\cong \mathbb C[z,z^{-1}]$, this leads to (\ref{kac}).

The {\it automorphic Lie algebras} are natural generalisations of twisted loop algebras, where $\mathbb C^*$ is replaced by a Riemann surface $X$ and the cyclic group $\zn{m}$ is replaced by a discrete group $\rg$ acting by Lie algebra automorphisms on $\mathfrak g$ and by holomorphic automorphisms on $X.$ 
The study of these Lie algebras  for $X$ being the punctured Riemann sphere was initiated by Lombardo and Mikhailov \cite{lombardo2004reductions, lombardo2005reduction}, 
who were motivated by the theory of integrable systems. The systematic classification for the same $X$ and $\mathfrak g=\slnc[2]$ was carried out by Lombardo and Sanders \cite{lombardo2010on},
and for $\mathfrak g=\slnc$ by Knibbeler, Lombardo and Sanders \cite{knibbeler2017higher}.

It is interesting to note that the first non-loop example of automorphic Lie algebras was introduced by Onsager in his pioneering  1944 paper \cite{onsager1944crystal} 
on the exact solution of the famous Ising model, although it was recognised only relatively recently \cite[Proposition 1]{roan1991onsager}. 
The Onsager Lie algebra corresponds to $X=\mathbb C^*, \, \mathfrak g=\slnc[2]$ and $\rg=\zn{2}$ acting on $X$ with the involution $z \mapsto z^{-1}$ and on $\slnc[2]$ with any involution. 

The Onsager algebra also provides a simple example for which the analogue of (\ref{kac}) does not hold: \[(\mf{g} \otimes_{\mathbb C} \mb{C}[z,z^{-1}])^\rg\ncong \mf{g} \otimes_{\mathbb C} \mb{C}[z,z^{-1}]^\rg.\] To see this, one can for instance check that the right hand side is perfect, whereas the left hand side is not. More generally, for a punctured compact Riemann surface $X$ we have $(\mf{g} \otimes_{\mathbb C} \mc{O}_X)^\rg\ncong \mf{g}\otimes_{\mathbb C}\mc{O}_X^\rg$ whenever $X\rightarrow \rg\backslash X$ contains an orbifold point $x_0$, granted $\rg_{x_0}$ acts nontrivially on $\mf{g}$ (see  \cite{duffield2020wild} for details).  
In this sense, (\ref{kac}) is a rare phenomenon.

In this paper we consider the first hyperbolic examples of automorphic Lie algebras, when $\rg$ is a finite index subgroup of the modular group $\SLNZ[2]$, acting on the upper half plane $X=\uh=\{\tau \in \mathbb C: \im \,\tau>0\}$ by M\H{o}bius transformations \[\begin{pmatrix}a&b\\c&d\end{pmatrix}\tau=\frac{a\tau+b}{c \tau+d},\quad \begin{pmatrix}a&b\\c&d\end{pmatrix}\in\rg,\quad \tau\in\uh\] and on $\mf{g}$ via some representation $\rho:\rg\rightarrow\Aut{\mf{g}}.$

The classical theory (see e.g. \cite{bruinier2008the,rankin1977modular}) suggests that in this case one should consider not only the modular functions $f(\tau)$, but also the modular forms $f(\tau)\dk{k}.$

Let $\mfpz=\bigoplus_{k\in\mb{Z}}\mfpk{k}$ be the algebra of the holomorphic modular forms for the group $\rg$, and $\mfz=\bigoplus_{k\in\mb{Z}}\mfk{k}$ be the algebra of the corresponding {\it weakly holomorphic} modular forms with possible poles at the cusps.

Consider the following Lie algebra of  $\mf{g}$-valued modular forms 
\[\aliapz=\bigoplus_{k\in\mb{Z}}\aliapk{k}\]
where the weight $k$ subspace $\aliapk{k}$ consists of the holomorphic maps $f: \uh \to \mf{g}$ satisfying the property
\beq{modk}
 f\left(\gamma\tau\right)=(c\tau+d)^k \rho(\gamma) f(\tau), \quad \forall \gamma=\left( \begin{array}{cc}
a&b\\
c&d
\end{array}\right)\in \rg,\, \forall \tau\in\uh,
\eeq
and growing at most polynomially at the cusps (see Section \ref{sec:vector} for the precise definition).
The Lie algebra structure is defined by the pointwise Lie bracket  $[f,g](\tau)=[f(\tau),g(\tau)]$ and is naturally $\mb{Z}$-graded. 

We consider also the Lie algebra of  {\it weakly holomorphic} $\mf{g}$-valued modular forms
$$
\aliaz=\bigoplus_{k \in \mathbb Z}\aliak{k}
$$
by allowing exponential growth at every cusp (see Section \ref{sec:vector}). Its subalgebra $\alia$ of elements of weight zero can be viewed as a straightforward analogue of the automorphic Lie algebras from the previous literature and will be of particular interest. We call the corresponding Lie algebras {\it automorphic Lie algebras of modular type}.
Note that if we ignore the Lie algebra structure, then $\aliapz$ and  $\aliaz$ will simply become the spaces of certain vector-valued modular forms, going back to Shimura and Selberg and studied more recently, in particular, by Knopp and Mason (see more history in \cite{gannon2014the}).

In this paper we 
focus our investigation to {\it automorphic Lie algebras of restricted modular type} when the representation $\rho:\rg\rightarrow\Aut{\mf{g}}$ is restricted from a representation of the Lie group $\SLNC[2]$. 

Our main result is the proof of the following isomorphism. Let, as before, $\mf{g}$ be a complex finite-dimensional simple Lie algebra and $G$ be a corresponding Lie group, which in this case can be chosen as the connected component $\Aut{\mf{g}}^0$ of the automorphism group of $\mf{g}$.

For every group embedding $\bar{\rho}:\SLNC[2]\to G$, we have the natural restricted homomorphism $\rho:\rg\to\Aut{\mf{g}}$ for any subgroup $\rg\subset\SLNZ[2]$. 
Note that any such embedding naturally induces a $\mb{Z}$-grading on $\mf{g}.$

\begin{Theorem}
\label{thm:main introduction}
For any finite index subgroup $\rg\subset\SLNZ[2]$ and restricted homomorphism $\rho:\rg\to\Aut{\mf{g}}$, we have the following isomorphism of the $\mb{Z}$-graded Lie algebras and $\mb{Z}$-graded $\mfpz$-modules
\beq{isom10}
\aliapz \cong  \mf{g} \otimes_\mathbb C \mfpz
\eeq
where the grading of $\mf{g}$ in the right hand side is induced from $\rho.$

In the weakly holomorphic case we have an isomorphism of  $\mb{Z}$-graded Lie algebras  and $\mb{Z}$-graded $\mfz$-modules
\beq{isom20}
\aliaz \cong  \mf{g} \otimes_\mathbb C \mfz.
\eeq
\end{Theorem}
This result can be considered as a hyperbolic version of the isomorphism (\ref{kac}) of Kac, since the action of $\Gamma$ on $\mathfrak g$ is inner.
It reduces the current problem to the classical problem of describing $\mfpz$ and $\mfz$. 

It is known that for any congruence subgroup of $\SLNZ[2]$ the graded ring $\mfpz$ is finitely generated  with generators of weight at most 6 and with relations in weight at most 12, see \cite{voight2015canonical, landesman2016spin}.

 In particular, Wagreich
\cite{wagreich1980algebras} 
showed that this ring is generated by two elements only for the modular group $\Gamma(1)=\SLNZ[2]$ itself and its congruence subgroups
$$
\Gamma_0(2)=\left\{\left( \begin{array}{cc}
a&b\\
c&d
\end{array}\right) \in \SLNZ[2]: c \equiv 0 \mod{2}\right\},
$$
$$
\Gamma(2)=\left\{\left( \begin{array}{cc}
a&b\\
c&d
\end{array}\right) \in \SLNZ[2]: \left( \begin{array}{cc}
a&b\\
c&d
\end{array}\right)\equiv \left( \begin{array}{cc}
1&0\\
0&1
\end{array}\right) \mod{2}\right\}.
$$
The first two groups play a fundamental role in the classical theory of elliptic functions. Namely, we have two canonical forms of elliptic curves due to Weierstrass and Jacobi:
$$y^2=4x^3-g_2x-g_3, \quad y^2=1-2\delta x^2+\varepsilon x^4.$$ The coefficients $g_2$ and $g_3$ are modular forms of $\Gamma(1)$ of weight 4 and 6, while Jacobi coefficients $\delta$ and $\varepsilon$ are modular forms of $\Gamma_0(2)$ of weight 2 and 4. The corresponding graded rings for these groups are
$\mfpz[\mg]=\mathbb C[g_2,g_3]$ and $\mfpz[\Gamma_0(2)]=\mathbb C[\delta,\varepsilon]$ respectively.
The ring $\mfpz[\Gamma(2)]$ is freely generated by the following two modular forms of weight 2: 
$$
F_2(\tau)=2E_2(2\tau)-E_2(\tau), \quad H_2(\tau)=F_2(\tau/2),
$$
where $E_2$ is the classical Eisenstein series, see the next section.

The automorphic Lie algebra $\alia$ depends nontrivially on the the choice of $\rho$ and the structure of $\mfz$, and is not in general isomorphic to $\mf{g} \otimes_\mathbb C \mfk{0}$. 
In particular, for the full modular group $\Gamma=\Gamma(1)=\SLNZ[2]$ and $\mf{g}=\slnc[2]$ we have  $\mfk{0}= \mathbb C[j],$ where $j$ is the {\it Hauptmodul (or Klein's absolute invariant)} \cite{rankin1977modular}, and the following result.
\begin{Theorem}
\label{thm:main 2 introduction}
For $\mf{g}=\slnc[2]$ and the adjoint action of $\rg=\SLNZ[2]$ on $\mf{g}$, we have the following isomorphism
$$
\alia\cong \mathbb C\langle h,e,f\rangle \otimes_\mathbb C \mathbb C[j],
$$
where 
$$
[h,e]=2 e, \quad [h,f]=-2 f, \quad [e,f]=j(j-1728) h.
$$
\end{Theorem}
As a corollary we show that in this case the automorphic Lie algebra $\alia$  is rather surprisingly isomorphic to the Onsager algebra.

The structure of the paper is as follows.

In the next section we consider the simplest example of our automorphic Lie algebras when $\rg=\SLNZ[2]$ is the modular group itself acting by conjugation on $\mathfrak g=\slnc[2]$.
Then we discuss the theory of vector-valued modular forms (VVMF) for all finite index subgroups of modular groups and the particular class of representations restricted from irreducible representations of $\SLNC[2].$ We describe the space of the corresponding VVMF using a special intertwining operator $\modaut{n}$ (see Theorem 3.2).
In Sections \ref{sec:automorphic} and \ref{sec:zero} we use all this to prove our main results on the structure of automorphic Lie algebras of restricted modular type.
In the zero weight case, the Lie algebra structures of $\aliapk{0}$ and $\aliak{0}$ essentially depend on the algebraic properties of the modular forms $\mfpz$ and $\mfz$. 
In Section \ref{sec:full} we consider in more details the automorphic Lie algebra $\aliak[{\mg}]{0}$ for the full modular group using the language of root cohomology introduced in \cite{knibbeler2020cohomology}. 
In Section \ref{sec:principal} we discuss the case of the genus zero principal congruence subgroups $\Gamma(N), \, 2\leq N\leq 5$, which are known after Klein to be closely related to the regular polyhedra. In Section \ref{sec:gamma6} the genus 1 cases of $\Gamma(6)$ and related modular commutator subgroup are discussed in relation to the Markov Diophantine equation.
We finish with a brief discussion of the extensions and representations of our algebras.

\section{The Simplest Modular Group Case}
\label{sec:simplest}
In this section we consider the simplest case when $\mathfrak g=\slnc[2]$ and  $\rg=\SLNZ[2]$ is the modular group \footnote{Sometimes the modular group is defined as $\PSL(2, \mathbb Z)=\SLNZ[2]/\pm \Id$, which is more natural from the hyperbolic geometry point of view, but we will follow the classical arithmetic tradition here.} acting 
on $\mf{g}$ by conjugation, which will be assumed throughout this section.

We will need the following deep results going back to the classical paper by Ramanujan \cite{ramanujan1916on}.
Recall that the Eisenstein series are defined by
\beq{def1}
E_k(\tau)=\frac{1}{2}\sum\frac{1}{(m\tau+n)^{k}}, \quad k>2,\, \im \tau>0,
\eeq
where the summation goes over coprime pairs $m,n$ of integers (see e.g. \cite{bruinier2008the}).
Their Fourier coefficients have a deep arithmetic meaning: for even $k$
$$
E_{k}(\tau)=1-\frac{2k}{B_{k}}\sum_{n=1}^\infty \sigma_{k-1}(n)q^n, \quad q=e^{2\pi i \tau},
$$
where $B_k$ are Bernoulli numbers and $\sigma_m(n)$ is the sum of $m$-th powers of all the divisors of $n$.

The last sum converges for $k=2$, which allows to define the Eisenstein series $E_2,$ a special case since the initial summation (\ref{def1}) for $k=2$ is divergent.
The special role of the Eisenstein series $E_2$ also becomes apparent in its transformation properties. It is not a modular form but it transforms as
\beq{eq:e2transformation}
E_2(\tau+1)=E_2(\tau),\quad E_2(-\tau^{-1})=\tau^2E_2(\tau)+\frac{12\tau}{2\pi i}
\eeq
and for any $\gamma\in\rg$ as $E_2(\gamma \tau)=(c\tau+d)^2E_2(\tau)+\frac{12}{2\pi i}c(c\tau+d)$.

Geometrically, it defines the canonical connections (also known as {\it modular}, or {\it Serre derivatives} \cite{serre1973congruences}) on the space of modular forms of weight $k$ by 
\newcommand{\D}{\frac{1}{2\pi i}\frac{\rd}{\rd\tau}}
\beq{ser}
D_k=\D-\frac{k}{12}E_2.
\eeq
This observation essentially goes back to Ramanujan \cite{ramanujan1916on}, who found a remarkable closed system of differential equations for the Eisenstein series $E_2, E_4, E_6$
\beq{ram1}
\D E_2=\!\frac{E_2^2-E_4}{12},\quad \D E_4=\!\frac{E_2E_4-E_6}{3},\quad \D E_6=\!\frac{E_2E_6-E_4^2}{2},
\eeq
which can be re-written as
$$
D_1 E_2=-\frac{E_4}{12},\quad D_4 E_4=-\frac{E_6}{3},\quad D_6 E_6=-\frac{E_4^2}{2},
$$
(the initial mismatch of subscripts is intended), 
see \cite[Proposition 15, p.49]{bruinier2008the}.

It is well-known that the ring of modular forms for $\rg=\SLNZ[2]$ is freely generated by the Eisenstein series $E_4$ and $E_6$, which are different from the Weierstrass parameters $g_2$ and $g_3$ only by a constant factor, see \cite{bruinier2008the}.
Ramanujan's relations imply that the quasi-modular extension $\mb{C}[E_2,E_4,E_6]$ of the ring of modular forms $\mfpz=\mb{C}[E_4,E_6]$ is closed under the usual derivatives.

We consider the standard {\it modular discriminant} $$\Delta=\frac{E_4^3-E_6^2}{1728}=q\prod_{n=1}^\infty(1-q^n)^{24}$$ and the {\it $j$-invariant} (which is a particular Hauptmodul)
$$
j=\frac{E_4^{3}}{\Delta}=\frac{1}{q}+744+196884q+21493760q^2+\dots,
$$
(see e.g. \cite{bruinier2008the}). 
Note that $\Delta$ vanishes at the cusp and $$j-1728=E_6^2/\Delta.$$
The algebra of weakly holomorphic forms $\mfz[{\rg}]$ is the localisation of $\mfpz$ at the cusp: 
\beq{whol}
\mfz[{\rg}]=\mb{C}[E_4,E_6,\Delta^{-1}]
\eeq
with $\mfk{0}= \mathbb C[j].$

We are now ready to study our Lie algebras. Let us start with an observation that
\beq{formg}
a_{-2}(\tau)=\begin{pmatrix}\tau&-\tau^2\\1&-\tau\end{pmatrix}
\eeq
belongs to our space $\aliapvar{-2}{\mf{g}}{\rg}{\rho}$ of forms of weight $-2,$
which follows from two simple identities
$$
a_{-2}(T(\tau))=\begin{pmatrix}\tau+1&-(\tau+1)^2\\1&-\tau-1\end{pmatrix}=T\begin{pmatrix}\tau&-\tau^2\\1&-\tau\end{pmatrix}T^{-1},
$$
$$
a_{-2}(S(\tau))=\begin{pmatrix}-\tau^{-1}&-\tau^{-2}\\1&\tau^{-1}\end{pmatrix}=\tau^{-2}S\begin{pmatrix}\tau&-\tau^2\\1&-\tau\end{pmatrix}S^{-1}, 
$$
where 
$$T=\begin{pmatrix}1&1\\0&1\end{pmatrix}, \, S=\begin{pmatrix}0&-1\\1&0\end{pmatrix}
.$$
Note that $a_{-2}(\tau)$ does not look like a typical modular form depending polynomially on $\tau$, but this phenomenon is well known in the theory of vector-valued modular forms, see Knopp and Mason \cite[Theorem 2.2]{knopp2011logarithmic} and Section \ref{sec:vector} below.

The form $a_{-2}(\tau)$ plays a similar role here as the form $\mathfrak{A}$ in \cite[Theorem 2.1]{lombardo2010on}, and it is interesting to view this in the context of the `duality' between modular forms and binary forms discussed in \cite{olver2000transvectants}.

By taking Serre derivatives we produce two more forms $a_0=D_{-2}\,a_{-2}$ and $a_2=D_0\,a_{0}$ of weights 0 and 2 respectively. 

Due to Knopp and Mason \cite[Theorem 3.13]{knopp2011logarithmic} we know that $\aliapz$ is a free module over $\mb{C}[E_4,E_6]$. From the results of Franc and Mason \cite{franc2018on} we can derive that the corresponding Hilbert series is
$$
\frac{t^{-2}+1+t^2}{(1-t^4)(1-t^6)}.
$$ 
Since $a_{-2}(\tau), \, a_{0}(\tau), \, a_{2}(\tau)$ are linearly independent over $\mb{C}[E_4,E_6]$ (see \cite[Lemma 3.1]{franc2018on}), it follows from the last formula that $\{a_{-2}, a_0, a_2\}$ is a basis for $\aliapvar{\ast}{\mf{g}}{\rg}{\rho}$ as a $\mb{C}[E_4,E_6]$-module.

Computing their commutators, we find that the matrix of $\ad(a_0)$ with respect to this basis has the form  \[\frac{1}{2\pi i} \begin{pmatrix}-2&0&0\\0&0&0\\E_4/18&0&2\end{pmatrix}.\] 
This motivates introducing the new basis 
\beq{efh}
f=a_{-2},\quad h=2\pi i\, a_0,\quad e=\pi^2/36\, E_4\, a_{-2}+2\pi^2\, a_2,
\eeq
or explicitly
\begin{align*}
f(\tau)&=\!\left(\begin{array}{rr}
\tau & -\tau^{2} \\
1 & -\tau
\end{array}\right),
\\h(\tau)&=\!\left(\begin{array}{rr}
\frac{1}{3} i \, \pi \tau E_{2}\left(\tau\right) + 1 & -\frac{1}{3} i \, \pi \tau^{2} E_{2}\left(\tau\right) - 2 \, \tau \\
\frac{1}{3} i \, \pi E_{2}\left(\tau\right) & -\frac{1}{3} i \, \pi \tau E_{2}\left(\tau\right) - 1
\end{array}\right),
\\e(\tau)&=\!\left(\begin{array}{rr}
\frac{1}{36} \, \pi^{2} \tau E_{2}\left(\tau\right)^{2} - \frac{1}{6} i \, \pi E_{2}\left(\tau\right) & -\frac{1}{36} \, \pi^{2} \tau^{2} E_{2}\left(\tau\right)^{2} + \frac{1}{3} i \, \pi \tau E_{2}\left(\tau\right) + 1 \\
\frac{1}{36} \, \pi^{2} E_{2}\left(\tau\right)^{2} & -\frac{1}{36} \, \pi^{2} \tau E_{2}\left(\tau\right)^{2} + \frac{1}{6} i \, \pi E_{2}\left(\tau\right)
\end{array}\right).
\end{align*}
The remarkable fact is that they form a standard $\mf{g}$-triple $$[h,e]=2e, \quad [h,f]=-2f, \quad [e,f]=h,$$ 
where the last relation can be checked by direct calculation.
Thus we have the following result, which is the simplest nontrivial case of Theorem \ref{thm:main introduction}.

\begin{Theorem}
\label{thm:main simplest}
For $\mathfrak g=\slnc[2], \, \Gamma=\SLNZ[2]$ and the adjoint action $\rho$ of $\Gamma \subset \SLNC[2]$ we have 
$$
\aliapvar{\ast}{\mathfrak g}{\Gamma}{\rho} =\mb{C}\langle h,e,f\rangle\otimes_\mathbb C\mb{C}[E_4,E_6]. 
$$
In the weakly holomorphic case we have 
$$
\aliaz =  \mb{C}\langle h,e,f\rangle\otimes_\mathbb C \mb{C}[E_4,E_6,\Delta^{-1}].
$$
The weight of $h$ is $0$, and $e$ and $f$ have weights $2$ and $-2$ respectively.
\end{Theorem}

Since all $\slnc[2]$-triples in $\slnc[2]$ are conjugate it is instructive to look for a conjugation sending 
$$\left(\begin{pmatrix}1&0\\0&-1\end{pmatrix},\begin{pmatrix}0&1\\0&0\end{pmatrix},\begin{pmatrix}0&0\\1&0\end{pmatrix}\right)\mapsto (h(\tau),e(\tau),f(\tau)).$$
\begin{proposition}
\label{prop:conjugation}
Define the matrix
$$
\modaut{1}
=\begin{pmatrix}2\pi i \tau\frac{E_2(\tau)}{12}+1&\tau\\2\pi i \frac{E_2(\tau)}{12}&1\end{pmatrix}.
$$
Then
\begin{align*}
h(\tau)&=\modaut{1}\begin{pmatrix}1&0\\0&-1\end{pmatrix}\modaut{1}^{-1},
\\e(\tau)&=\modaut{1}\begin{pmatrix}0&1\\0&0\end{pmatrix}\modaut{1}^{-1},
\\f(\tau)&=\modaut{1}\begin{pmatrix}0&0\\1&0\end{pmatrix}\modaut{1}^{-1}.
\end{align*}
\end{proposition}
The computation of the structure of automorphic Lie algebras using Serre derivatives is hard to generalise as the Serre derivatives do not respect the Lie bracket. Instead, we will develop a general theory starting from Proposition \ref{prop:conjugation}.

Before proceeding with the general theory, let us study the zero weight automorphic Lie algebra $\aliak{0}$ in this simplest case.

Duke and Jenkins studied a basis of $\mfk[{\rg}]{k}$ with particularly interesting $q$-expansions \cite{duke2008on}. A preliminary result of theirs is the following.
\begin{lemma}[Duke-Jenkins]
\label{lem:meromorphic weight k}
If $k$ is odd then $\mfk[{\rg}]{k}=\{0\}$. If $k$ is even, then $\mfk[{\rg}]{k}$ is the one-dimensional module over $\mb{C}[j]$,  generated by \[F_{k}=\Delta^{\ell}E_s\] with $\ell$ and $s$ uniquely defined by $k=12\ell+s$ and $s\in\{0,4,6,8,10,14\}$.
\end{lemma}
For our purposes it is more illustrative to write the generator in the form
\[F_{k}=\Delta^{\ell}E_4^{\n}E_6^{\nn}\]
where $\n$ and $\nn$ are the unique positive integers such that $$4\n+6\nn=s.$$
To see that this is the same function, notice that $E_4^{\n}E_6^{\nn}$ and $E_s$ are elements of the one-dimensional vector space $\mfpk{s}$ and the first term in their $q$-expansion is $1$.

Lemma \ref{lem:meromorphic weight k} and Theorem \ref{thm:main simplest} prove Theorem \ref{thm:main 2 introduction}, which we can now restate in terms of explicit matrices.
\begin{Theorem}
\label{thm:alias sl2 weight zero}
Define the matrices
$$
\bar{f}=\Delta^{-1}E_4^2 E_6\,f,\quad\bar{h}=h, \quad\bar{e}=\Delta^{-1}E_4 E_6\,e,
$$
where
$f, h, e$ are given by (\ref{efh}). 
Then the zero weight automorphic Lie algebra
\beq{isom22}
\aliavar{0}{\mf{g}}{\rg}{\rho}= \mb{C}\langle \bar{h},\bar{e},\bar{f}\rangle\otimes_\mathbb C\mb{C}[j],
\eeq
and $\bar h, \bar e, \bar f$
satisfy the commutation relations 
\beq{barrel}
[\bar{h},\bar{e}]=2\bar{e}, \quad
[\bar{h},\bar{f}]=-2\bar{f}, \quad
[\bar{e},\bar{f}]=j(j-1728)\, \bar{h}.
\eeq
\end{Theorem}

Note that the Lie algebra $\aliapvar{\ast}{\mf{g}}{\rg}{\rho}$ is perfect in the sense that it coincides with its commutator, while for $\aliavar{0}{\mf{g}}{\rg}{\rho}$ the abelianisation is $2$-dimensional.

Note also that (\ref{barrel}) describes a flat deformation of Lie algebras depending on parameter $j\in\mb{C}$. It is isomorphic to $\mf{g}$ at generic points, while at orbifold points $j=0$ and $j=1728$ we have the contraction to the solvable Lie algebra of the group of plane isometries: 
$$[\bar{h},\bar{e}]=2\bar{e}, \quad [\bar{h},\bar{f}]=-2\bar{f}, \quad [\bar{e},\bar{f}]=0.$$

It is interesting that the formulae (\ref{isom22}), (\ref{barrel}) are similar to the Riemann sphere case, see Lombardo and Sanders \cite{lombardo2010on}.

Remarkably, this algebra is isomorphic to the Onsager algebra.
The algebra $\On$ was introduced in \cite{onsager1944crystal} as the Lie algebra with complex basis $A_k,\,G_m$, $k\in\mb{Z}, m\in\mb{N}$ and commutation relations
\beq{eq:Onsager}
 \begin{array}{ll}
[G_m,G_n]&=0
\\{[}G_m,A_k]&=2A_{k+m}-2A_{k-m}
\\{[}A_k,A_\ell]&=G_{k-\ell}
\end{array}
\eeq
with $G_{-m}=-G_m$ and $G_0=0$.

\begin{Theorem}
For $\mathfrak g=\slnc[2], \, \Gamma=\SLNZ[2]$ and the adjoint action $\rho$ of $\Gamma \subset \SLNC[2]$ the zero weight automorphic Lie algebra $\aliavar{0}{\mf{g}}{\rg}{\rho}$ is isomorphic to the Onsager algebra. 
An isomorphism $\mathfrak O\to\alia$ is determined by $$A_0\mapsto \bar{h}, \quad A_1\mapsto\frac{(2j-1728)\bar{h}-2\bar{e}+2\bar{f}}{1728}.$$
\end{Theorem}

\begin{proof}
It was shown by Roan \cite{roan1991onsager} that the Onsager algebra is isomorphic to the subalgebra of $\slnc[2]\otimes_{\mathbb C}\mathbb C[z,z^{-1}]$ consisting of all fixed points of the involution
\[M(t)\mapsto \begin{pmatrix}0&1\\1&0\end{pmatrix}M(1/z)\begin{pmatrix}0&1\\1&0\end{pmatrix}.\]
Notice that this Lie algebra does not belong to the class of twisted loop algebras used in the construction of affine Kac-Moody algebras. It does however belong to the class of automorphic Lie algebras \cite{lombardo2005reduction}.

Roan's presentation of $\On$ is the following. Let $$\theta=\Ad\begin{pmatrix}0&i\\i&0\end{pmatrix}\in\Aut{\slnc[2]}$$ and define $\theta_0\in\Aut{\slnc[2]\otimes_{\mathbb C}\mb{C}[z,z^{-1}]}$ by $$\theta_0 (A\otimes p(z))=\theta A\otimes p(z^{-1}).$$ 
Denote the automorphic Lie algebra of fixed points by $\mathfrak A=\left(\slnc[2]\otimes_{\mathbb C}\mb{C}[z,z^{-1}]\right)^{\theta_0}$.
Then one can check that the linear map $\phi:\On\to \mathfrak A$ defined by
$$
\phi(A_k)= \begin{pmatrix}0&z^k\\z^{-k}&0\end{pmatrix}, \quad \phi(G_m)= (z^m-z^{-m})\begin{pmatrix}1&0\\0&-1\end{pmatrix}
$$
for $k\in\mb{Z}$ and $m\in\mb{N}$, is an isomorphism of these Lie algebras.

Moreover, the following elements
$$
\hat e =\frac{z-z^{-1}}{8}\begin{pmatrix}1&-1\\1&-1\end{pmatrix}, \quad
\hat f = \frac{z-z^{-1}}{8}\begin{pmatrix}1&1\\-1&-1\end{pmatrix}, \quad
\hat h=\begin{pmatrix}0&1\\1&0\end{pmatrix}
$$
generate the algebra $\mathfrak A$ as $\mathbb C[\hat{j}]$-module and satisfy the relations
$$
[\hat {h},\hat{e}]=2\hat{e}, \,\, 
[\hat {h},\hat{f}]=-2\hat{f}, \,\,[\hat{e},\hat{f}]=\hat{j}(\hat{j}-1)\, \hat{h},
$$
where
$$
\hat{j}=\frac{z^2+2+z^{-2}}{4}
$$
is a Hauptmodul of the group $\mathbb Z/2\mathbb Z$.

Thanks to Theorem \ref{thm:alias sl2 weight zero} we now find an isomorphism $\mathfrak{A}\to\alia$ by scaling, and thus we obtain an isomorphism $\mathfrak{O}\to\alia$. Since the Lie algebra $\mathfrak{O}$ is generated by $A_0$ and $A_1$, this isomorphism is defined by the image of these elements.
\end{proof}

To study the general case we need some results from the theory of vector-valued modular forms.

\section{Vector-Valued Modular Forms}
\label{sec:vector}

The history of vector-valued modular forms could be traced back to Poincar\'{e} \cite{poincare1882memoire}, but it became of significance after the work of Selberg and Shimura (see the history in Gannon \cite{gannon2014the}).

Vector-valued modular forms (VVMF) are defined as follows.
Consider a representation $\rho:\rg\rightarrow\GL(V)$ of a subgroup $\rg$ of $\SLNZ$. A holomorphic map $f:\mb{H}\rightarrow V$ is an \emph{unrestricted vector-valued modular form} of weight $k$ if 
\begin{equation}
\label{eq:equivariance2}
f(\gamma \tau)=(c\tau+d)^k\rho(\gamma)f(\tau),\quad \forall \gamma=\begin{pmatrix}a&b\\c&d\end{pmatrix}\in\rg,\,\forall \tau\in\uh.
\end{equation}
Alternatively, we can use the notation $f(\tau)\dk{k}$ where (\ref{eq:equivariance2}) is an equivariance condition (see \cite{gunning1962lectures}).

We consider the two subclasses of unrestricted VVMF, denoted $\vvmfk{k}$ and $\vvmfpk{k}$ respectively, one with at most exponential growth at the cusps, and the other with at most polynomial growth at the cusps, in the following sense.

Any cusp $s\in\mb{Q}$ can be mapped to $i\infty$ by an element $g=\begin{pmatrix}a&b\\c&d\end{pmatrix}$ of $\SLNZ[2]$ which allows one to define the parameter 
$$\tau_s=\frac{a\tau+b}{c\tau+d}\in\uh$$ 
which approaches $i\infty$ when $\tau$ approaches the cusp $s$. Note that $\tau_s$ is defined up to addition of integers.

We say that $f:\uh\to V$ has \emph{exponential growth} at a cusp $s$ with parameter $\tau_s$ if 
\beq{eq:taue}
||f(\tau)||<C e^{M|\tau_s|}
\eeq 
for some $M$, and all $\tau_s$ with imaginary part large enough. 
Here $||\cdot ||$ is a norm on the vector space $V$. Its choice is not important since all norms are known to be equivalent in finite dimension.

Similarly, we say that $f$ has \emph{polynomial growth} at the cusp $s$ if it satisfies 
\beq{eq:taup}
||f(\tau)||<C|\tau_s|^N
\eeq 
for some $N$, and all $\tau_s$ with imaginary part large enough.

Usual (scalar-valued) modular forms with polynomial growth are actually modular forms which are holomorphic at the cusps, but this is not true in the vector-valued case.
Indeed, the simplest nontrivial example of a vector-valued modular form is 
\beq{eq:tau1}
\begin{pmatrix}\tau\\1\end{pmatrix}\dmk{1}
\eeq
which is an element of $\vvmfpk{-1}$ where $\rho$ is the natural representation of $\SLNZ[2]$ (see Shimura \cite{shimura1959sur}).

\newcommand{\coc}{\nu}
The component functions of VVMF are logarithmic $q$-series; we make this precise in the following proposition, which is a slight generalisation of a result of Knopp and Mason \cite[Sections 2.2 and 2.3]{knopp2011logarithmic}.

Let $\gamma\in\rg$ be a parabolic element fixing the cusp $s$ such that $\gamma=g^{-1} T^m g$ for some $g\in \SLNZ[2]$, $\tau_s=g\tau$ and $q_s=e^{2\pi i \tau_s/m}$ be the corresponding cusp parameter.
\begin{proposition}
Let $f=(f_1,\ldots,f_d)$ be a vector-valued modular form on $\rg$ of weight $k$. 
Then at any cusp $s$ any component $f_i(\tau)$ has the following convergent logarithmic $q$-series
\beq{eq:series in tau_s}
f_i(\tau)\dk{k}=\sum_{\ell=0}^{N}\tau_s^\ell h_{i,\ell}(q_s)\rd\tau_s^{\frac{k}{2}},\quad  h_{i,\ell}(q)=\sum_{n\in\mb{Z}}a_{i,\ell,n} q^n
\eeq
where $N=d-1$.
The form $f$ has the exponential growth at cusp $s$ if and only if for all $i=1,\ldots,d$ and $\ell=0\ldots,N$ the coefficients $a_{i,\ell,n}=0$ for all $n<M$, while the polynomial growth corresponds to $a_{i,\ell,n}=0$ for all $n<0$.
\end{proposition}
\begin{proof}
The forms \[f_i(g^{-1}\tau)\rd(g^{-1}\tau)^{\frac{k}{2}},\quad i=1\ldots,d\] span a $T^m$-invariant vector space due to modularity (\ref{eq:equivariance2}) of $f$:
\begin{align*}
f(g^{-1}T^{m}\tau)
=f(\gamma g^{-1}\tau)
=\rho(\gamma)f(g^{-1}\tau).
\end{align*}
Therefore we can apply \cite[Theorem 2.2]{knopp2011logarithmic} and conclude that these components have convergent logarithmic $q$-series
\[
f_i(g^{-1}\tau)\rd(g^{-1}\tau)^{\frac{k}{2}}=\sum_{\ell=0}^{N}\tau^\ell h_{i,\ell}(q)\rd\tau^{\frac{k}{2}},\quad  h_{i,\ell}(q)=\sum_{n\in\mb{Z}}a_{i,\ell,n} q^n.
\]
If we now replace $\tau$ by $\tau_s=g\tau$ we find (\ref{eq:series in tau_s}), proving the first part.

To prove the second part about growth, consider the series (\ref{eq:series in tau_s}) and its shifts:
$$
f_i(\gamma^j\tau)\rd(\gamma^j\tau)^{\frac{k}{2}}=\sum_{\ell=0}^{N}(\tau_s+jm)^\ell h_{i,\ell}(q_s)\rd\tau_s^{\frac{k}{2}}, \quad j=0,\dots, N.
$$
This allows us to express each form $h_{i,\ell}(q_s)\rd\tau_s^{\frac{k}{2}}$ as a linear combination of $f_i(\gamma^j\tau)\rd(\gamma^j\tau)^{\frac{k}{2}}$, $j=0,\dots, N$, with coefficients polynomial in $\tau_s.$
This means that the exponential growth of $f_i(\tau)$ at $\tau=s$ is equivalent to the exponential growth of $h_{i,\ell}(q_s)$ for all $\ell$ at $q_s=0$.

It is clear that if $a_{i,\ell,n}=0$ for all $n<M$ for some $M$, then $h_{i,\ell}(q)=\sum_{n\in\mb{Z}}a_{i,\ell,n} q^n$ has the exponential growth as function of $\tau$.
Conversely, suppose that there exists $\ell$ such that $a_{i,\ell,n}\neq 0$ for infinitely many negative $n.$ Then $q=0$ is an essential singularity of $h_{i,\ell}(q)$ and by the Great Picard Theorem \cite{conway1978functions} at each neighbourhood of $q=0$, $\,h_{i,\ell}$ assumes each complex number with only one possible exception, which contradicts the exponential growth in $\tau.$

The case of the polynomial growth is similar.
\end{proof}

Notice that the corresponding spaces of VVMF $\vvmfz=\bigoplus_{k\in\mb{Z}} \vvmfk{k}$ and $\vvmfpz=\bigoplus_{k\in\mb{Z}} \vvmfpk{k}$ constitute graded modules over the graded algebras of classical modular forms $\mfz$ and $\mfpz$.
Recall that a module $M$ over a graded algebra $A=\bigoplus_{k\ge0}A_k$ is graded if there is a vector space direct sum decomposition $M=\bigoplus_{\ell\in\mb{Z}}M_\ell$ with $M_\ell=0$ for all $\ell<-N$ for some $N\in\mb{N}$
such that $A_kM_\ell\subset M_{k+\ell}$.
The Hilbert series of $M$ is the generating function
\[H(M,t)=\sum_{k\in\mb{Z}}\dim M_k t^k.\]

Consider now a particular class of representations of $\Gamma \subset \SLNC[2]$, the restrictions from the irreducible representations of $\SLNC[2]$, which are all known to be the symmetric powers of the natural representation $V=\mb{C}^2$ of $\SLNC[2]$. 
Let \[\rho_n:\rg\rightarrow \SL(V^n)\] be the corresponding representation, where $V^n$ is the $n+1$-dimensional vector space $\Sym^n \mb{C}^2$, which can be realised as the space of binary forms of degree $n$.

$V^n$-valued modular forms were already introduced by Shimura in 1959 \cite{shimura1959sur} and further studied by Kuga and Shimura \cite{kuga1960on} (see also recent work by Zemel \cite[Proposition 3.1]{zemel2015on}). We will produce a different description with different proof, more suitable for our purposes.

The representations $\rho_n$ of $\rg$ extend to representations of the Lie group \[\bar{\rho}_n:\SLNC[2]\rightarrow \SL(V^n).\] This enables us to use the powerful techniques of Lie theory. In particular, we have the corresponding representations of the Lie algebra 
\[\rd\bar{\rho}_n:\slnc[2]\rightarrow \mf{sl}(V^n)\]
at our disposal, where the exponential map $\exp:\mf{sl}(V)\rightarrow\SL(V)$ intertwines the two: $\exp \rd\bar{\rho}_n=\bar{\rho}_n\exp$. 
For notational convenience we will define
\[H=\rd\bar{\rho}_n\left(\begin{pmatrix}1&0\\0&-1\end{pmatrix}\right),\quad E=\rd\bar{\rho}_n\left(\begin{pmatrix}0&1\\0&0\end{pmatrix}\right),\quad F=\rd\bar{\rho}_n\left(\begin{pmatrix}0&0\\1&0\end{pmatrix}\right). \]
The space $V^n$ decomposes into the direct sum of $1$-dimensional eigenspaces $V^n_k$ of $H$ with eigenvalue $k\in\{-n,-n+2,\ldots,n\}$. This direct sum defines a grading on $V^n$.

Let $\hol$ be the space of holomorphic $V^n$-valued forms on $\uh$ and define $\modaut{n}$ to be the linear endomorphism of $\hol$ given by
\beq{eq:modaut}
\modaut{n}=\exp\left(\tau E\right)\exp\left(2\pi i \frac{E_2(\tau)}{12} F\right)\exp\left(\ln(\dk{1}) H\right).
\eeq
Here the expression $\exp\left(\ln(\dk{1}) H\right)$ is defined as the linear endomorphism of $\hol$, multiplying the eigenspace of $H$ with eigenvalue $k$ by $\dk{k}$.
 Note that Kuga and Shimura used the left factor $\exp\left(\tau E\right)$ in their work \cite{kuga1960on}.

Our key observation
is that the operator $\modaut{n}$
sends any vector of (scalar-valued) modular forms to a vector-valued modular form, and all $V^n$-valued modular form are obtained this way.

\newcommand{\holtriv}{\hol_{\id}}
\newcommand{\holrho}{\hol_{\rho}}
Note that on the space $\hol$ we have two $\rg$-module structures, corresponding to two different actions on $V^n$: the trivial action and the action defined by $\rho_n$. 
We claim that the operator $\modaut[]{n} :\hol\to\hol$ is an intertwiner between these two actions, which preserves the polynomial growth.
 
\begin{Theorem} 
\label{thm:isomorphism vvmf}
The operator $\modaut[]{n}$ establishes an isomorphism of graded $\mfpz[\rg]$-modules
$$V^n\otimes_{\mathbb C} \mfpz \xrightarrow{\modaut[]{n}}\vvmfpz[\rho_n].
$$

In particular, $\vvmfpz[\rho_n]$ is a free $\mfpz[\rg]$-module of dimension $n+1$ with 
Hilbert series
\begin{equation}
\label{genf}
H(\vvmfpz[\rho_n],t)=(t^{-n}+t^{-n+2}+\ldots+t^n)H(\mfpz,t).
\end{equation}
\end{Theorem}

\begin{remark}
The freeness problem of the modules of vector-valued modular forms for wide class of Fuchsian groups and their representations is studied in important papers by Marks and Mason \cite{marks2010structure}, Gannon \cite{gannon2014the}, Candelori and Frank \cite{candelori2019vector} and Gottesman \cite{gottesman2020the}. Our result only provides an elementary proof in the simplest situation.
\end{remark}

\begin{proof} We use the following key Lemma.
\begin{lemma}For all $n\in\mb{N}$, $\gamma\in\SLNZ[2]$ and $\tau\in\uh$ we have
\label{lem:modaut}
\begin{equation}
\label{klemma}
\modaut[(\gamma\tau)]{n}=\rho_n(\gamma)\modaut{n}.
\end{equation}
\end{lemma}
\begin{proof} We first prove the relation (\ref{klemma}) for $n=1$. Indeed the transformation property (\ref{eq:e2transformation}) of the Eisenstein series $E_2(\tau)$ says that for $\gamma=\begin{pmatrix}a&b\\c&d\end{pmatrix}$
 \[2\pi i\frac{E_2(\gamma\tau)}{12}= 2\pi i \frac{E_2(\tau)}{12}(c\tau+d)^2+c(c\tau+d).\]
Using this we have
{
\begin{align*}
\modaut[(\gamma\tau)]{1}&=
\begin{pmatrix}1&\gamma\tau\\0&1\end{pmatrix}
\begin{pmatrix}1&0\\2\pi i \frac{E_2(\gamma\tau)}{12}&1\end{pmatrix}
\begin{pmatrix}\rd(\gamma\tau)^{\frac{1}{2}}&0\\0&\rd(\gamma\tau)^{-\frac{1}{2}}\end{pmatrix}
\\&=
\begin{pmatrix}1&\frac{a\tau+b}{c\tau+d}\\0&1\end{pmatrix}
\begin{pmatrix}1&0\\2\pi i \frac{E_2(\tau)}{12}(c\tau+d)^2+c(c\tau+d)&1\end{pmatrix}\begin{pmatrix}(c\tau+d)^{-1}\rd\tau^{\frac{1}{2}}&0\\0&(c\tau+d)\rd\tau^{-\frac{1}{2}}\end{pmatrix}
\\&=
\begin{pmatrix}1&\frac{a\tau+b}{c\tau+d}\\0&1\end{pmatrix}
\begin{pmatrix}1&0\\2\pi i \frac{E_2(\tau)}{12}(c\tau+d)^2+c(c\tau+d)&1\end{pmatrix}
\begin{pmatrix}(c\tau+d)^{-1}&0\\0&c\tau+d\end{pmatrix}
\begin{pmatrix}\rd\tau^{\frac{1}{2}}&0\\0&\rd\tau^{-\frac{1}{2}}\end{pmatrix}
\\&=
\begin{pmatrix}1&\frac{a\tau+b}{c\tau+d}\\0&1\end{pmatrix}
\begin{pmatrix}(c\tau+d)^{-1}&0\\0&c\tau+d\end{pmatrix}
\begin{pmatrix}1&0\\2\pi i \frac{E_2(\tau)}{12}+c(c\tau+d)^{-1}&1\end{pmatrix}
\begin{pmatrix}\rd\tau^{\frac{1}{2}}&0\\0&\rd\tau^{-\frac{1}{2}}\end{pmatrix}
\\&=
\begin{pmatrix}1&\frac{a\tau+b}{c\tau+d}\\0&1\end{pmatrix}
\begin{pmatrix}(c\tau+d)^{-1}&0\\0&c\tau+d\end{pmatrix}
\begin{pmatrix}1&0\\c(c\tau+d)^{-1}&1\end{pmatrix}
\begin{pmatrix}1&0\\2\pi i \frac{E_2(\tau)}{12}&1\end{pmatrix}
\begin{pmatrix}\rd\tau^{\frac{1}{2}}&0\\0&\rd\tau^{-\frac{1}{2}}\end{pmatrix}
\\&=
\begin{pmatrix}1&\frac{a\tau+b}{c\tau+d}\\0&1\end{pmatrix}
\begin{pmatrix}(c\tau+d)^{-1}&0\\0&c\tau+d\end{pmatrix}
\begin{pmatrix}1&0\\c(c\tau+d)^{-1}&1\end{pmatrix}
\begin{pmatrix}1&-\tau\\0&1\end{pmatrix}
\modaut{1}
\\&=
\begin{pmatrix}(1+c(a\tau+b))(c\tau+d)^{-1}&(a\tau+b)-\tau(1+c(a\tau+b))(c\tau+d)^{-1}\\c&(c\tau+d)-c\tau\end{pmatrix}
\modaut{1}
\\&=
\begin{pmatrix}a&b\\c&d\end{pmatrix}
\modaut{1}=\rho_1(\gamma)\modaut{1}.
\end{align*}
}

Now we apply the homomorphism $\bar{\rho}_n$ to see that 
$$\modaut[(\gamma\tau)]{n}
=\bar{\rho}_n(\modaut[(\gamma\tau)]{1})
=\bar{\rho}_n(\gamma\modaut{1})
=\bar{\rho}_n(\gamma)\bar{\rho}_n(\modaut{1})
={\rho}_n(\gamma)\modaut{n},$$
which completes the proof of Lemma \ref{lem:modaut}.
\end{proof}

Observe now that $\modaut[]{n}$ preserves the grading, the polynomial growth and is invertible, which implies the first part of the claim.

The Hilbert series of the vector-valued and scalar-valued modular forms are related by (\ref{genf})
which can be verified by mapping a basis $v_{-n}\otimes1,v_{-n+2}\otimes1,\ldots,v_n\otimes1$ of $V^n\otimes \mfpz$, where $H v_k=kv_k$, to a basis of $\vvmfpz[\rho_n]$, as free $\mfpz$-modules.
\end{proof}

\begin{remark}
The right column of $\modaut[]{n}$ is $(\tau^n,\tau^{n-1},\ldots,1)^t\dmk{n}$, which is a VVMF playing a central role in the work of Shimura \cite{shimura1959sur}. It is interesting that it is holomorphic on the plane $\mb{C}$, rather than just the half plane. In contrast, there is no scalar modular form that can be extended holomorphically to any real number: the real line is a natural boundary. 

Knopp and Mason studied this phenomenon \cite{knopp2013vector} and found that if we restrict to the representations where $\rho(T)$ has eigenvalues with absolute value one, then $(\tau^n,\tau^{n-1},\ldots,1)\dmk{n}$ is, up to isomorphism, the only VVMF that does not have the real line as a natural boundary.
In that sense it is exceptional.
\end{remark}

\section{Automorphic Lie Algebras of Restricted Modular Type}
\label{sec:automorphic}

Let now $\rg \subset \SLNZ[2]$ be as before a finite index subgroup of the modular group, and consider the representations $\rho: \rg \rightarrow\Aut{\mf{g}}$ restricted from the representations $\bar{\rho}:\SLNC[2]\rightarrow\Aut{\mf{g}}.$ 

Such representations are related to the embeddings of $\SLNC[2]$ into automorphism groups of simple Lie algebras, which are classified. 
The problem is equivalent to the classification of nilpotent orbits: orbits of nilpotent elements in $\mf{g}$ under the action of the connected component $G=\Aut{\mf{g}}^0$ of the automorphism group. 
The latter classification is well described and listed in \cite{collingwood1993nilpotent}. We first describe the equivalence between the two classifications and then briefly recap the theorems of Jacobson-Morozov and of Kostant that enable the classification.

To classify Lie group morphisms $\bar{\rho}:\SLNC[2]\rightarrow\Aut{\mf{g}}$ when $\mf{g}$ is semisimple, it is sufficient to classify Lie algebra morphisms $\phi:\slnc[2]\rightarrow{\mf{g}}$. Indeed, since $\SLNC[2]$ is connected and $\bar{\rho}$ continuous, one only needs to consider maps $\bar{\rho}:\SLNC[2]\rightarrow\Aut{\mf{g}}^0=G$. The tangent map is $\rd\bar{\rho}:\slnc[2]\rightarrow T_1G$, and by semisimplicity of $\mf{g}$ we have $T_1G \cong \mf{g}$ \cite[Propositions 1.120 and 1.121]{knapp2002lie}. The exponential map intertwines these maps in the sense that $\bar{\rho}(\exp(X))=\exp(\rd\bar{\rho}(X))$. Since $\exp(\slnc[2])$ generates $\SLNC[2]$, we see that $\bar{\rho}$ is determined by its derivative $\rd\bar{\rho}$. Vice versa, any homomorphism $\phi:\slnc[2]\rightarrow{\mf{g}}$ defines a homomorphism $\bar{\rho}:\SLNC[2]\rightarrow\Aut{\mf{g}}$ by $\bar{\rho}(\exp(X))=\exp(\ad(\phi(X)))$.

The $G$-orbits of $\Hom(\slnc[2],\mf{g})$ are in turn in one-to-one correspondence with nilpotent orbits, by the following theorems, as outlined in \cite[Section 3.2]{collingwood1993nilpotent}. A standard triple will be a triple of elements $(H,E,F)$ of a Lie algebra with Lie brackets $[H,E]=2E$, $[H,F]=-2F$ and $[E,F]=H$, hence spanning a subalgebra isomorphic to $\slnc[2]$. The element $E$ is the nilpositive element of the triple.
\begin{Theorem}[Jacobson-Morozov] Let $\mf{g}$ be a complex semisimple Lie algebra. If $X$ is
a nonzero nilpotent element of $\mf{g}$, then it is the nilpositive element of a standard
triple. Equivalently, for any nilpotent element $X$, there exists a homomorphism $\phi:\slnc\rightarrow\mf{g}$ such that $\phi\left(\begin{pmatrix}0&1\\0&0\end{pmatrix}\right)= X$.
\end{Theorem}
\begin{Theorem}[Kostant] Let $\mf{g}$ be a complex semisimple Lie algebra. Any two standard
triples $(H,E, F)$ and $(H',E, F')$ with the same nilpositive element are conjugate
by an element of $\Aut{\mf{g}}^0$.
\end{Theorem}
If $\mc{N}$ is the set of $G$-orbits of nilpotent elements of $\mf{g}$, then there is a map
\[G\backslash\Hom(\slnc[2],\mf{g})\rightarrow\mc{N},\quad G\cdot\phi\mapsto G\cdot\phi\left(\begin{pmatrix}0&1\\0&0\end{pmatrix}\right).\]
It is clear that this map is well defined. It is surjective by the Theorem of Jacobson-Morozov. It is injective by the Theorem of Kostant.

The classification of nilpotent orbits $\mc{N}$ is described in \cite{collingwood1993nilpotent} (see the lists in Section 8.4 for all exceptional Lie algebras). At the same time, this classifies the embeddings $\bar{\rho}:\SLNC[2]\rightarrow\Aut{\mf{g}}$ needed for the current research.

Let $\bar{\rho}:\SLNC[2]\rightarrow\Aut{\mf{g}}$ and $\rd\bar{\rho}:\slnc[2]\rightarrow\Der{\mf{g}}$ be a representation and its derivative. Introduce an analogue of operator (\ref{eq:modaut}) as
\beq{eq:modautlie}
\modaut{}=\exp\left(\tau E\right)\exp\left(2\pi i \frac{E_2(\tau)}{12} F\right)\exp\left(\ln(\dk{1}) H\right),
\eeq
\beq{eq:HEF}
H=\rd\bar{\rho}\left(\begin{pmatrix}1&0\\0&-1\end{pmatrix}\right),\quad E=\rd\bar{\rho}\left(\begin{pmatrix}0&1\\0&0\end{pmatrix}\right),\quad F=\rd\bar{\rho}\left(\begin{pmatrix}0&0\\1&0\end{pmatrix}\right).
\eeq 
\begin{Theorem}
\label{thm:aliahz}
The operator $\modaut[]{}$ given by (\ref{eq:modautlie}) establishes an isomorphism of graded Lie algebras and graded $\mfpz$-modules
\beq{isom11}
\mf{g}\otimes_{\mathbb C}\mfpz\xrightarrow{\modaut[]{}}\aliapz
\eeq
where $\mf{g}$ is endowed with the grading defined by the eigenvalues of $H$.

Similarly, in the weakly holomorphic case we have the isomorphism of the Lie algebras and $\mb{Z}$-graded $\mfk{0}$-modules
\beq{isom21}
\mf{g}\otimes_{\mathbb C}\mfz\xrightarrow{\modaut[]{}}\aliaz.
\eeq
\end{Theorem}

\begin{proof}
Indeed, by Lemma  \ref{lem:modaut} the operator $\modaut[]{}$ intertwines the action of $\Gamma$ on both sides, preserves the growth condition and is invertible.

If we ignore the Lie algebra structure, then we can write $\bar{\rho}=\bigoplus_{i}\bar{\rho}_{n_i}$ and thus $\modaut{}=\bigoplus_{i}\modaut{n_i},$
so the claim also follows from Theorem \ref{thm:isomorphism vvmf}.
\end{proof}
Theorem \ref{thm:aliahz} can be viewed as a ``modular analogue'' of the following result of Kac (see \cite[Proposition 8.5]{kac1990infinite}). 
In Theorems \ref{thm:alias odd group} and \ref{thm:alias even group}, we will prove an even closer analogue of Kac's result concerning functions rather than forms.
\begin{Theorem}[Kac]
Let $\sigma$ be an inner automorphism of order $n$ of a simple finite-dimensional Lie algebra $\mf{g}$ of the form
\[\sigma=\exp\left(\frac{2\pi i}{n}\ad(h)\right).\]
Then the automorphism $\modaut[(z)]{}$ of the Lie algebra $\mf{g}\otimes_{\mathbb C}\mb{C}[z,z^{-1}]$ defined by \[\modaut[(z)]{}=\exp(\ln(z)\ad(h))\] 
establishes an isomorphism between $\mf{g}\otimes_{\mathbb C}\mb{C}[z^n,z^{-n}]$ and the twisted Lie algebra $L(\mf{g},\sigma)$.
\end{Theorem}
The proof follows from the relation $\modaut[(e^{2\pi i/n} z)]{}=\sigma\modaut[(z)]{}$, which is analogous to our relation $\modaut[(\gamma \tau)]{}=\rho(\gamma)\modaut[(\tau)]{}$ from Lemma \ref{lem:modaut}. 

\section{The Zero Weight Case}
\label{sec:zero}

In contrast to $\aliapz$ and $\aliaz$, the Lie algebra structures of the weight zero subalgebras $\aliapk{0}$ and $\aliak{0}$ depend on the structure of $\mfpz$ and $\mfz$ respectively (more so than merely as module), and thus resist uniform description for general subgroups $\rg$ of $\SLNZ[2]$. Nonetheless, if $\mfpz$ is understood for a specific group, Theorem \ref{thm:aliahz} can be used to give explicit structure constants. 

Theorem \ref{thm:aliahz} shows that $\aliaz\cong \mf{g}\otimes_{\mathbb C} \mfz$ as graded $\mfz$-modules and graded Lie algebras. The Lie algebra grading on $\mf{g}$ is given by its weight decomposition $\mf{g}=\bigoplus_{k\in\mb{Z}} \mf{g}_{k}$ as $\SLNC[2]$-module, so as a corollary we have
\beq{corol}
\aliak{0}\cong\bigoplus_{k\in\mb{Z}} \mf{g}_{-k}\otimes_{\mathbb C} \mfk{k}.
\eeq
By the same reasoning we also have $\aliapk{0}=\bigoplus_{k\in\mb{Z}} \mf{g}_{-k}\otimes_{\mathbb C} \mfpk{k}.$
We will describe the structure of these Lie algebras in more detail, starting with the case of polynomial growth. Since $\mfpk{k}$ is finite-dimensional, $\mfpk{k}=\{0\}$ for $k<0$ and $\mfpk{0}$ consists only of constant functions, for any finite index subgroup $\rg$ of $\SLNZ[2]$ \cite{gunning1962lectures}, we see that the automorphic Lie algebra is the complex finite-dimensional graded Lie algebra
$$\aliapk{0}=\mf{g}_{0}\oplus\bigoplus_{n\in\mb{N}} \mf{g}_{-n}\otimes_{\mathbb C} \mfpk{n}$$
(where we identify $\mf{g}_{0}$ with constant functions $\uh\to\mf{g}_0$).
The subalgebra $\mf{g}_{0}=\{A\in\mf{g}\,|\,HA=0\}$ is reductive , i.e. of the form $\mf{g}_{0}=\mf{z}(\mf{g}_0)\oplus \mf{g}_0'$ where $\mf{z}(\mf{g}_0)$ is the centre of $\mf{g}_0$ and the derived subalgebra $\mf{g}_0'=[\mf{g}_{0},\mf{g}_{0}]$ is semisimple (see e.g. \cite[Lemma 2.1.2]{collingwood1993nilpotent}). Thus we find
\begin{Theorem}
The automorphic Lie algebra $\aliapk{0}$ of modular type with polynomial growth is a complex finite-dimensional Lie algebra with Levi decomposition
$$\aliapk{0}=\mf{r}\oplus\mf{s}$$
where the radical $\mf{r}=\mf{z}(\mf{g}_0)\oplus\bigoplus_{n\in\mb{N}} \mf{g}_{-n}\otimes_{\mathbb C} \mfpk{n}$ and the Levi subalgebra
$\mf{s}=\mf{g}_{0}'.$
\end{Theorem}
\begin{proof}
We have already established the vector space direct sum $\aliapk{0}=\mf{r}\oplus\mf{s}$. What remains to be shown is that $\mf{r}$ is the radical of $\aliapk{0}$.

It is clear that $\mf{r}$ is a solvable ideal of $\aliapk{0}$ due to the grading and the fact that $\mf{z}(\mf{g}_0)$ is central in $\mf{g}_0$. Hence $\mf{r}$ is contained in the radical of $\aliapk{0}$.

The quotient map $\aliapk{0}\to \aliapk{0}/\mf r$ restricts to an isomorphism $\mf s \to \aliapk{0}/\mf r$. The latter has trivial radical, hence the radical of $\aliapk{0}$ is contained in $\mf{r}$.
\end{proof}
We pause the study of the automorphic Lie algebras with polynomial growth at this point and proceed to the case $\aliak{0}$ of exponential growth.

A subgroup $\rg$ of $\SLNZ[2]$ is called {\it even} if it contains $-\Id$, otherwise $\rg$ is called {\it odd}.
Note that $\Gamma(2)$ is even and $\Gamma(N)$ is odd when $N\ge 3$. 
If $\rg$ is even then $\mfk{k}=\{0\}$ for odd $k$, since $f\in\mfk{k}$ implies $f(\tau)=f(-\Id(\tau))=(-1)^kf(\tau)$.

In the case where $\rg$ is even, it is convenient to make the assumption that $\rho(-\Id)=\id$ and we will do so below. Otherwise we have to replace $\mf{g}$ by $\mf{g}^{\rho(\pm\Id)}$ accordingly.

We have the following simple observation.
\begin{lemma}
\label{lem:mf odd group}
The weakly holomorphic modular forms of $\rg$ are given by 
$$
\mfk{k}=\mfk{0}F_1^{k}
$$
for some $F_1$ independent of $k$ if and only if $\mfk{1}$ contains an element $F_1$ that is never zero on $\uh$.
\end{lemma}
\begin{proof}
If $\mfk{k}=\mfk{0}F_1^{k}$ then in particular $F_1^{-1}\in\mfk{-1}$, hence $F_1$ is never zero on $\uh$.
If, on the other hand, $F_1\in \mfk{1}$ is never zero on $\uh$, then it is clear that $\mfk{0}F_1^{k}\subset\mfk{k}$ for any integer $k$. To prove the other inclusion, we take $f\in\mfk{k}$ and notice that $f/F_1^{k}\in\mfk{0}$ because $F_1$ can only have zeros at cusps. 
\end{proof}

\begin{Theorem}
\label{thm:alias odd group}
If $\rg$ is a finite index subgroup of $\SLNZ[2]$, and $\mfpk{1}$ contains an element that is never zero on $\uh$, then there is a Lie algebra isomorphism 
\beq{Psi0}
\alia\cong \mf{g}\otimes_{\mathbb C} \mfk{0}.
\eeq
\end{Theorem}
\begin{proof}Due to Lemma \ref{lem:mf odd group} we know that $\mfk{k}=\mfk{0}F_1^{k}$. When we plug this into (\ref{corol}) we have $$\alia\cong\left(\bigoplus_{k\in\mb{Z}} \mf{g}_{-k}\otimes_{\mathbb C} \mathbb C F_1^{k} \right)\otimes_{\mathbb C}\mfk{0}.$$ Now we notice that the subalgebra $\bigoplus_{k\in\mb{Z}}\mf{g}_{-k}\otimes_{\mathbb C}\mathbb C  F_1^{k}$ is isomorphic to $\mf{g}$ by $A_{-k}\otimes F_1^{k}\mapsto A_{-k}$ for $A_{-k}\in\mf{g}_{-k}$. This finishes the proof.
\end{proof}

We do not know which groups have everywhere nonvanishing modular form of weight $1$, but this is true for the following principal congruence subgroups.

\begin{Theorem}
The groups $\Gamma(3k)$, $\Gamma(4k)$ and $\Gamma(5k)$, with $k\in\mb{N}$, have a modular form of weight $1$ which is never zero on $\uh$, so if $\rg$ is one of these groups we do have the isomorphism $$\alia\cong \mf{g}\otimes_{\mathbb C} \mfk{0}.$$
\end{Theorem}
\begin{proof}
The principal congruence subgroups $\Gamma(N)$ with $N=3,4,5$ are related to regular polyhedra (see Section \ref{sec:principal} below).
An explicit expression for the nonvanishing modular forms of weight $1$ for $\Gamma(N)$ with $N=3,4,5$ goes back to Klein.  They are conveniently provided in Schultz's lecture notes \cite{schultz2015notes} (see Proposition 4.9.6 there and Section \ref{sec:principal} below).

Let $\eta(\tau)$ be {\it Dedekind's eta-function} defined by
\beq{Dedekind}
\eta(\tau)=q^{\frac{1}{24}}(q; q)_\infty, \quad q=\exp(2\pi i \tau)
\eeq
where
$$
(x;q)_\infty:=\prod_{k=0}^{\infty}(1-xq^k).
$$
Introduce also the {\it Klein forms}
$\mf{k}_{r_1,r_2}$ by
\beq{Klein}
\mf{k}_{r_1,r_2}(\tau)=q_z^{(r_1-1)/2}\frac{(q_z;q)_\infty(q/q_z;q)_\infty}{(q;q)_\infty^2}
\eeq
with $q_z=\exp(2\pi i z), \,\,z=r_1\tau+r_2$.

From \cite[Proposition 4.9.6]{schultz2015notes} we extract
\begin{proposition} 
The following explicit expressions give the modular forms of weight 1
\[\frac{\eta(3\tau)^3}{\eta(\tau)},\quad \frac{\eta(4\tau)^4}{\eta(2\tau)^2},\quad \frac{\eta(5\tau)^{15}\mf{k}_{\frac{1}{5},\frac{0}{5}}(5\tau)^5}{\eta(\tau)^3}\]
for the principal congruence subgroups $\Gamma(3)$, $\Gamma(4)$ and $\Gamma(5)$ respectively.
\end{proposition}

It is easy to see that the corresponding modular forms have no zeros in $\uh,$ so we can apply our Theorem \ref{thm:alias odd group}. Since $\Gamma(MN)\subset \Gamma(N)$ for any $M\in\mb{N}$, the claim now follows. 
\end{proof}

To deal with $\Gamma(2)$ we need the following
\begin{lemma}
\label{lem:mf even weight}
The weakly holomorphic modular forms of $\rg$ of even weight $k=2\ell$ are given by 
$$
\mfk{k}=\mfk{0}F_2^{\ell}
$$
for some $F_2$ independent of $k$ if and only if $\mfk{2}$ contains an element $F_2$ that is never zero on $\uh$.
\end{lemma}
\begin{proof}
If $\mfk{k}=\mfk{0}F_2^{\ell}$ then in particular $F_2^{-1}\in\mfk{-2}$ hence $F_2$ is never zero on $\uh$.
If, on the other hand, $F_2\in \mfk{2}$ is never zero on $\uh$, then it is clear that $\mfk{0}F_2^{\ell}\subset\mfk{k}$ for any integer $k$. To prove the other inclusion, we take $f\in\mfk{k}$ and notice that $f/F_2^{\ell}\in\mfk{0}$ because $F_2$ can only have zeros at cusps. 
\end{proof}

\begin{Theorem}
\label{thm:alias even group}
If $\rg$ is an even finite index subgroup of $\SLNZ[2]$, and $\mfpk{2}$ contains an element that is never zero on $\uh$, and $\rho(-\Id)=\id$, then there is a Lie algebra isomorphism 
\beq{Psi}
\alia\cong \mf{g}\otimes_{\mathbb C} \mfk{0}.
\eeq
In particular, this applies to $\Gamma=\Gamma(2)$.
\end{Theorem}
\begin{proof}
First of all, we know that $\mfk{k}=\{0\}$ when $k$ is odd because $\Gamma$ is even. Due to Lemma \ref{lem:mf even weight} we know that $\mfk{k}=\mfk{0}F_2^{k/2}$ when $k$ is even. Substituting into (\ref{corol}) gives $$\alia\cong\left(\bigoplus_{k\textrm{ even}} \mf{g}_{-k}\otimes_{\mathbb C}\mathbb C F_2^{k/2} \right)\otimes_{\mathbb C}\mfk{0}.$$ 
Now we notice that the subalgebra $\bigoplus_{k\textrm{ even}}\mf{g}_{-k}\otimes_{\mathbb C}\mathbb C F_2^{k/2}$ is isomorphic to $\bigoplus_{k\textrm{ even}}\mf{g}_{-k}$ by $A_{-k}\otimes F_2^{k/2}\mapsto A_{-k}$ for $A_{-k}\in\mf{g}_{-k}$. By the assumption that $\rho(-\Id)=\id$ we have $\bigoplus_{k\textrm{ even}}\mf{g}_{-k}=\mf{g}$, which establishes the first claim.

This can be applied to $\Gamma(2)$ since it has 3 non-vanishing modular forms $\theta_i^4,\, i=2,3,4$ of weight 2 with $\theta_i$ being the classical theta-series (see Section \ref{sec:principal} for the details).
\end{proof}

\section{Full Modular Group Case and Root Cohomology}
\label{sec:full}

Let us describe the corresponding algebras $\aliak[{\mg}]{0}$ for the full modular group $\Gamma(1)=\SLNZ[2]$ in more detail. In this case, we do not have an analogue of Theorems \ref{thm:alias odd group} and \ref{thm:alias even group}. We will again assume that $\rho$ is restricted from $\SLNC[2]$ and $\rho(-\Id)=\id$.

To find the structure constants of $\aliak[{\mg}]{0}$ we will use Lemma \ref{lem:meromorphic weight k} and the language of root cohomology introduced in \cite{knibbeler2020cohomology}. 

It is well known that any semisimple element in a simple Lie algebra is contained in a Cartan subalgebra. Thus for any embedding of $\SLNC[2]$ in the automorphism group of $\mf{g}$ we may choose a Cartan subalgebra $\mf{h}$ of $\mf{g}$ such that $H$ is in $\ad \mf{h}$. A Chevalley basis $\{H_i, A_\alpha, i=1\ldots,N, \alpha\in\roots\}$ 
\[
\begin{array}{rll}
{[}H_i, H_j]&\backspace=0&
\\{[}H_i, H_\alpha]&\backspace=\alpha(H_i) A_\alpha&
\\{[}A_\alpha, A_{-\alpha}]&\backspace=\sum n_iH_i,& \alpha=\sum n_i\alpha_i,
\\{[}A_\alpha,A_\beta]&\backspace=\epsilon(\alpha,\beta)A_{\alpha+\beta},& \alpha+\beta\in\roots,
\\{[}A_\alpha,A_\beta]&\backspace=0,& \alpha+\beta\notin\roots\cup\{0\},
\end{array} 
\]
(see for instance \cite[Section 25]{humphreys1972introduction}) of $\mf{g}$ relative to such a Cartan subalgebra consists of $H$ eigenvectors with integral eigenvalues. If $k(\alpha)$ is the $H$-eigenvalue of $X_\alpha$ then $k$ is an additive map on the root system $\roots$. 
In fact, the classification of $\slnc[2]$-triples in $\mf{g}$ is listed in \cite{collingwood1993nilpotent} using Dynkin diagrams with labels $0$, $1$ and $2$. The map $k$ is the additive extension of these labels to the whole root system.

We revisit Theorem \ref{thm:aliahz} and its corollary (\ref{corol}) in this special case and include the finer structure provided by the Chevalley basis.  By Theorem \ref{thm:aliahz}, the elements
\[h_i(\tau)=\modaut{} H_i\otimes1,\quad  a_\alpha(\tau)=\modaut{} A_\alpha\otimes\dk{k(\alpha)},\quad i=1\ldots,N,\quad \alpha\in\roots,\]
where $\modaut{}$ is given by (\ref{eq:modautlie}),
form a basis of $\aliapz[\mg]$ as free $\mb{C}[E_4,E_6]$-module with identical structure constants as $\{H_i, A_\alpha\}$. 

Recall the functions $F_k=\Delta^\ell E_4^{\n} E_6^{\nn}$ of Lemma \ref{lem:meromorphic weight k} which generate the module $\mfk[\Gamma(1)]{k}$ over $\mb{C}[j]$. The exponents $(\n,\nn)$ can be seen as the residue map $2\mathbb Z \rightarrow\{0,1,2\}\times \{0,1\}$ of the isomorphism of groups 
\[2\mb{Z}/12\mb{Z}\rightarrow\mb{Z}/3\mb{Z}\times \mb{Z}/2\mb{Z},\quad \bar{2}\mapsto(\overline{2},\overline{1}).\]
For convenience we list its $6$ values 
$$
\begin{tabular}{cccccccc} 
$k$&$-6$ &$-4$ &$-2$&$0$ &$2$&$4$ &$6$\\
$(\n,\nn)$&$(0,1)$&$(2,0)$&$(1,1)$&$(0,0)$&$(2,1)$&$(1,0)$&$(0,1)$.\\
\end{tabular}
$$
By Lemma \ref{lem:meromorphic weight k}, the elements
\[h_i(\tau),\quad \ba_\alpha(\tau)=F_{-k(\alpha)} a_\alpha(\tau)\]
for $i=1\ldots,N$, form a basis of $\alia[\mg]$ as free $\mb{C}[j]$-module (here we use the assumption that $\rho(-\Id)=\id$ so that $k(\alpha)$ is even for all $\alpha\in\roots$). 
 
Now we define 
\beq{eq:2cocycles}
\om=1/3\,\mathsf{d}(\n\circ -k), \quad  \omm=1/2\,\mathsf{d}(\nn\circ -k)
\eeq
where $\mathsf{d}$ is the coboundary operator defined on the functions on the root system by $\mathsf{d}\omega(\alpha,\beta)=\omega(\beta)-\omega(\alpha+\beta)+\omega(\alpha)$. 

The symmetric $2$-cocycles (\ref{eq:2cocycles}) are maps sending a pair of roots to $0$ or $1$, and can therefore conveniently be represented by a graph, whose set of vertices is the set of roots, and whose set of edges is, for instance, the collection of pairs of roots that are send to $1$ by the cocycle.

Tables \ref{tab:A2}, \ref{tab:B2}, and \ref{tab:G2} below show all such graphs for simple Lie types of rank $2$ and embeddings $\SLNC[2]\to\Aut{\mf{g}}$ with $-\Id\mapsto \id$. This describes all corresponding automorphic Lie algebras $\alia[\mg]$ thanks to Theorem \ref{thm:aliam}.

We should note that there are no obstructions to describing Lie algebras of high rank. Indeed, $k$ is an additive function on the root system and therefore defined by its values on the simple roots. The same holds for $\om$ and $\omm$.

\begin{minipage}{0.3\textwidth}
\begin{table}[H]
\caption{$A_2$}
\label{tab:A2}
\begin{center}
\begin{tabular}{cc} \hline\\
principal
\\\\\hline\\
\begin{tikzpicture}[scale=\dynkintablescale,baseline=(current bounding box.center), font=\dynkinfont,decoration={
    markings,
        mark=between positions 0.6 and 5 step 8mm with {\arrow{<}} }]
\path 
node at ( -0.5,0) [dynkinnode,label=90: $2 $,label=270: $\alpha_1 $] (one) {$ $}
node at ( 0.5,0) [dynkinnode,label=90: $2 $,label=270: $\alpha_2 $] (two) {$ $}
;
\draw[] (one) to node []{$ $} (two);
\end{tikzpicture}
\\\\\hline\\
\begin{tikzpicture}[scale=\scaleAA]
  \path node at ( 0,0) [root,draw,label=270: $ $] (zero) {$ $}	
   node at ( 0,0) [2ndroot,draw,label=270: $ $] (zero) {$ $}	
	node at ( 4,0) [root,draw,label=0: $\alpha_1 $] (one) {$ $}
  	node at ( 2,3.464) [root,draw,label=60: $ $] (two) {$ $}
  	node at ( -2,3.464) [root,draw,label=180: $\alpha_2 $] (three) {$ $}
	node at ( -4,0) [root,draw,label=180: $ $] (four) {$ $}
	node at ( -2,-3.464) [root,draw,label=240: $ $] (five) {$ $}
	node at ( 2,-3.464) [root,draw,label=300: $ $] (six) {$ $};

	\draw[scochainp] (two.240-\dangle) to node []{$ $} (five.60+\dangle);
	\draw[scochainp] (three.300-\dangle) to node []{$ $} (six.120+\dangle);
	\draw[scochainm] (six.120-\dangle) to node []{$ $} (three.300+\dangle);
	\draw[scochainp] (four.0-\dangle) to node []{$ $} (one.180+\dangle);
	\draw[scochainm] (one.180-\dangle) to node []{$ $} (four.0+\dangle);

	\draw[scochainp] (two.210-\dangle) to node []{$ $} (four.30+\dangle);
	\draw[scochainp] (four.330-\dangle) to node [sloped,above]{$ $} (six.150+\dangle);
	\draw[scochainm] (six.150-\dangle) to node [sloped,above]{$ $} (four.330+\dangle);
	\draw[scochainp] (six.90-\dangle) to node [sloped,above]{$ $} (two.270+\dangle);
	\draw[scochainm] (three.330-\dangle) to node [sloped,above]{$ $} (one.150+\dangle);
\end{tikzpicture}
\end{tabular}
\end{center}
\end{table}
\end{minipage}
\begin{minipage}{0.7\textwidth}
\begin{table}[H]
\caption{$B_2$}
\label{tab:B2}
\begin{center}
\begin{tabular}{ccc} \hline\\
subregular&principal
\\\\\hline\\
\begin{tabular}{c}\begin{tikzpicture}[scale=\dynkintablescale,baseline=(current bounding box.center), font=\dynkinfont,decoration={
    markings,
        mark=at position 0.5 with {
\draw (-3pt,-3pt) -- (0pt,0pt);
\draw (0pt,-0pt) -- (-3pt,3pt);
} }]
\path 
node at ( 0,0) [sdynkinnode,label=90: $0 $,label=270: $\alpha_1 $] (one) {$ $}
node at ( 1,0) [ldynkinnode,label=90: $2 $,label=270: $\alpha_2 $] (two) {$ $}
;
\draw[] (one.30) to node []{$ $} (two.150);
\draw[] (one.-30) to node []{$ $} (two.-150);
\fill[postaction={decorate}] (two) -- (one);
\end{tikzpicture}\end{tabular}
&
\begin{tabular}{c}\begin{tikzpicture}[scale=\dynkintablescale,baseline=(current bounding box.center), font=\dynkinfont,decoration={
    markings,
        mark=at position 0.5 with {
\draw (-3pt,-3pt) -- (0pt,0pt);
\draw (0pt,-0pt) -- (-3pt,3pt);
} }]
\path 
node at ( 0,0) [sdynkinnode,label=90: $2 $,label=270: $\alpha_1 $] (one) {$ $}
node at ( 1,0) [ldynkinnode,label=90: $2 $,label=270: $\alpha_2 $] (two) {$ $}
;
\draw[] (one.30) to node []{$ $} (two.150);
\draw[] (one.-30) to node []{$ $} (two.-150);
\fill[postaction={decorate}] (two) -- (one);
\end{tikzpicture}\end{tabular}
\\\\\hline\\
\begin{tikzpicture}[scale=\scaleBB]
\path 	
node at ( 0,0) [root,draw,label=270: $ $] (zero) {$ $}	
node at ( 0,0) [2ndroot,draw,label=270: $ $] (zero) {$ $}	
node at (1,0) [root,draw,label=0: {$\alpha_1 $}] (m10) {$  $}
node at (-1,1) [root,draw,label=180: {$\alpha_2 $}] (m01) {$  $}
node at (0,1) [root,draw,label=90: {$ $}] (m11) {$  $}
node at (1,1) [root,draw,label=0: {$ $}] (m21) {$  $}

node at (-1,0) [root,draw,label=0: {$ $}] (n10) {$  $}
node at (1,-1) [root,draw,label=180: {$ $}] (n01) {$  $}
node at (0,-1) [root,draw,label=90: {$ $}] (n11) {$  $}
node at (-1,-1) [root,draw,label=0: {$ $}] (n21) {$  $}
;

\draw[scochainp](m01.315-\dangle) to node{}(n01.135+\dangle);
\draw[scochainm](m01.315+\dangle) to node{}(n01.135-\dangle);
\draw[scochainp](m11.270-\dangle) to node{}(n11.90+\dangle);
\draw[scochainm](m11.270+\dangle) to node{}(n11.90-\dangle);
\draw[scochainp](m21.225-\dangle) to node{}(n21.45+\dangle);
\draw[scochainm](m21.225+\dangle) to node{}(n21.45-\dangle);

\draw[scochainp](m01.300-\dangle) to node{}(n11.120+\dangle);
\draw[scochainm](m01.300+\dangle) to node{}(n11.120-\dangle);

\draw[scochainp](m11.300-\dangle) to node{}(n01.120+\dangle);
\draw[scochainm](m11.300+\dangle) to node{}(n01.120-\dangle);

\draw[scochainp](m11.255-\dangle) to node{}(n21.75+\dangle);
\draw[scochainm](m11.255+\dangle) to node{}(n21.75-\dangle);

\draw[scochainp](m21.255-\dangle) to node{}(n11.75+\dangle);
\draw[scochainm](m21.255+\dangle) to node{}(n11.75-\dangle);

\end{tikzpicture}
&
\begin{tikzpicture}[scale=\scaleBB]
\path 	
node at ( 0,0) [root,draw,label=270: $ $] (zero) {$ $}	
node at ( 0,0) [2ndroot,draw,label=270: $ $] (zero) {$ $}	
node at (1,0) [root,draw,label=0: {$\alpha_1 $}] (m10) {$  $}
node at (-1,1) [root,draw,label=180: {$\alpha_2 $}] (m01) {$  $}
node at (0,1) [root,draw,label=90: {$ $}] (m11) {$  $}
node at (1,1) [root,draw,label=0: {$ $}] (m21) {$  $}

node at (-1,0) [root,draw,label=0: {$ $}] (n10) {$  $}
node at (1,-1) [root,draw,label=180: {$ $}] (n01) {$  $}
node at (0,-1) [root,draw,label=90: {$ $}] (n11) {$  $}
node at (-1,-1) [root,draw,label=0: {$ $}] (n21) {$  $}
;
\draw[scochainp](m10.180+\dangle) to node{}(n10.-\dangle);
\draw[scochainm](m10.180-\dangle) to node{}(n10.\dangle);
\draw[scochainp](m01.315-\dangle) to node{}(n01.135+\dangle);
\draw[scochainm](m01.315+\dangle) to node{}(n01.135-\dangle);
\draw[scochainp](m11.270) to node{}(n11.90);
\draw[scochainm](m21.225) to node{}(n21.45);

\draw[scochainp](m10.135) to node{}(m11.-45);
\draw[scochainp](n10.45) to node{}(m11.225);
\draw[scochainp](n10.-45) to node{}(n11.135);
\draw[scochainp](m11.-60) to node{}(n01.120);
\draw[scochainp](n10.-30-\dangle) to node{}(n01.150+\dangle);
\draw[scochainm](n10.-30+\dangle) to node{}(n01.150-\dangle);

\draw[scochainm](m01.-30) to node{}(m10.150);
\draw[scochainm](m21.210) to node{}(n10.30);
\draw[scochainm](m10.210) to node{}(n21.30);

\end{tikzpicture}
\end{tabular}
\end{center}
\end{table}
\end{minipage}

\begin{center}
\begin{table}[H]
\caption{$G_2$}
\label{tab:G2}
\begin{center}
\begin{tabular}{cccccc} \hline\\
subregular&principal
\\\\\hline\\
\begin{tikzpicture}[scale=\dynkintablescale,baseline=(current bounding box.center), font=\dynkinfont,decoration={
    markings,
        mark=at position 0.5 with {
\draw (-3pt,-3pt) -- (0pt,0pt);
\draw (0pt,-0pt) -- (-3pt,3pt);
} }]
\path 
node at ( 1,0) [ldynkinnode,label=90: $2 $,label=270: $\alpha_1 $] (one) {$ $}
node at ( 2,0) [sdynkinnode,label=90: $0 $,label=270: $\alpha_2 $] (two) {$ $}
;
\draw[] (one) to node []{$ $} (two);
\draw[] (one.45) to node []{$ $} (two.135);
\draw[] (one.-45) to node []{$ $} (two.-135);
\fill[postaction={decorate}] (one) -- (two);
\end{tikzpicture}
&
\begin{tikzpicture}[scale=\dynkintablescale,baseline=(current bounding box.center), font=\dynkinfont,decoration={
    markings,
        mark=at position 0.5 with {
\draw (-3pt,-3pt) -- (0pt,0pt);
\draw (0pt,-0pt) -- (-3pt,3pt);
} }]
\path 
node at ( 1,0) [ldynkinnode,label=90: $2 $,label=270: $\alpha_1 $] (one) {$ $}
node at ( 2,0) [sdynkinnode,label=90: $2 $,label=270: $\alpha_2 $] (two) {$ $}
;
\draw[] (one) to node []{$ $} (two);
\draw[] (one.45) to node []{$ $} (two.135);
\draw[] (one.-45) to node []{$ $} (two.-135);
\fill[postaction={decorate}] (one) -- (two);
\end{tikzpicture}
\\\\\hline
\begin{tikzpicture}[scale=\scaleGG]
\path 
node at ( 0,0) [root,draw,label=270: $ $] (zero) {$ $}	
node at ( 0,0) [2ndroot,draw,label=270: $ $] (zero) {$ $}
node at (-1.5,.866) [root,draw,label=90: $\alpha_1 $] (m10) {$ $}
node at (1,0) [root,draw,label=0: $\alpha_2 $] (m01) {$ $}
node at (-.5,.866) [root,draw,label=90: $ $] (m11) {$ $}
node at (.5,.866) [root,draw,label=90: $ $] (m12) {$ $}
node at (1.5,.866) [root,draw,label=90: $ $] (m13) {$ $}
node at (0,2*.866) [root,draw,label=90: $ $] (m23) {$ $}
node at (1.5,-.866) [root,draw,label=270: $ $] (n10) {$ $}
node at (-1,0) [root,draw,label=180: $ $] (n01) {$ $}
node at (.5,-.866) [root,draw,label=270: $ $] (n11) {$ $}
node at (-.5,-.866) [root,draw,label=270: $ $] (n12) {$ $}
node at (-1.5,-.866) [root,draw,label=270: $ $] (n13) {$ $}
node at (0,-2*.866) [root,draw,label=270: $ $] (n23) {$ $}		
;
\draw[scochainp] (m11.270-\dangle) to node []{$ $} (n12.90+\dangle);
\draw[scochainp] (m12.270-\dangle) to (n11.90+\dangle);
\draw[scochainp] (n11.180+\dangle) to node []{$ $} (n12.-\dangle);

\draw[scochainp] (m23) to [bend right=\bend](n13);
\draw[scochainp] (m23) to [bend left=\bend](n10);
\draw[scochainp] (n10.180+\bend+\dangle) to [bend left=\bend](n13.-\bend-\dangle);

\draw[scochainp] (m10.-45-\dangle) to (n11.135+\dangle);
\draw[scochainp] (m23) to (n11);
\draw[scochainp] (m23) to (n12);
\draw[scochainp] (m13.225+\dangle) to (n12.45-\dangle);
\draw[scochainp] (n10.135+\dangle) to (m11.-45-\dangle);
\draw[scochainp] (n13.45-\dangle) to (m12.225+\dangle);

\draw[scochainp] (m10.-30-\dangle) to (n10.150+\dangle);
\draw[scochainp] (m11.-60-\dangle) to (n11.120+\dangle);
\draw[scochainp] (m12.240-\dangle) to (n12.60+\dangle);
\draw[scochainp] (m13.210+\dangle) to (n13.30-\dangle);
\draw[scochainp] (m23) to (n23);

\draw[scochainm] (m11.270+\dangle) to node []{$ $} (n12.90-\dangle);
\draw[scochainm] (m12.270+\dangle) to (n11.90-\dangle);

\draw[scochainm] (m12) to node []{$ $} (m11);
\draw[scochainm] (n11.180-\dangle) to node []{$ $} (n12.+\dangle);

\draw[scochainm] (m10) to [bend left=\bend](m13);
\draw[scochainm] (n10.180+\bend-\dangle) to [bend left=\bend](n13.-\bend+\dangle);

\draw[scochainm]  (m13.225-\dangle) to (n12.45+\dangle);
\draw[scochainm] (m12.225-\dangle) to (n13.45+\dangle);
\draw[scochainm] (m11.-45+\dangle) to (n10.135-\dangle);
\draw[scochainm] (m10.-45+\dangle) to (n11.135-\dangle);

\draw[scochainm] (m10.-30+\dangle) to (n10.150-\dangle);
\draw[scochainm] (m11.-60+\dangle) to (n11.120-\dangle);
\draw[scochainm] (m12.240+\dangle) to (n12.60-\dangle);
\draw[scochainm] (m13.210-\dangle) to (n13.30+\dangle);

\end{tikzpicture}
&
\begin{tikzpicture}[scale=\scaleGG]
\path 
node at ( 0,0) [root,draw,label=270: $ $] (zero) {$ $}	
node at ( 0,0) [2ndroot,draw,label=270: $ $] (zero) {$ $}
node at (-1.5,.866) [root,draw,label=90: $\alpha_1 $] (six) {$ $}
node at (1,0) [root,draw,label=0: $\alpha_2 $] (one) {$ $}
node at (-.5,.866) [root,draw,label=90: $ $] (five) {$ $}
node at (.5,.866) [root,draw,label=90: $ $] (three) {$ $}
node at (1.5,.866) [root,draw,label=90: $ $] (two) {$ $}
node at (0,2*.866) [root,draw,label=90: $ $] (four) {$ $}
node at (1.5,-.866) [root,draw,label=270: $ $] (twelve) {$ $}
node at (-1,0) [root,draw,label=180: $ $] (seven) {$ $}
node at (.5,-.866) [root,draw,label=270: $ $] (eleven) {$ $}
node at (-.5,-.866) [root,draw,label=270: $ $] (nine) {$ $}
node at (-1.5,-.866) [root,draw,label=270: $ $] (eight) {$ $}
node at (0,-2*.866) [root,draw,label=270: $ $] (ten) {$ $}		
;
\draw[scochainp] (one) to node []{$ $} (five);
\draw[scochainp] (eleven) to (seven);
\draw[scochainp] (seven) to node []{$ $} (five);

\draw[scochainp] (twelve.120-\bend+\dangle) to [bend right=\bend](four.-60+\bend-\dangle);
\draw[scochainp] (twelve) to [bend left=\bend](eight);
\draw[scochainp] (eight) to [bend left=\bend](four);

\draw[scochainp] (two) to (seven);
\draw[scochainp] (twelve.150+\dangle) to (seven.-30-\dangle);
\draw[scochainp] (twelve) to (five);
\draw[scochainp] (ten) to (five);
\draw[scochainp] (eight) to (one);
\draw[scochainp] (four) to (eleven);

\draw[scochainp] (two) to (eight);
\draw[scochainp] (one.180+\dangle) to (seven.0-\dangle);
\draw[scochainp] (eleven) to (five);
\draw[scochainp] (ten.90+\dangle) to (four.-90-\dangle);
\draw[scochainp] (twelve.150+\dangle) to (six.-30-\dangle);

\draw[scochainm] (three) to node []{$ $} (seven);
\draw[scochainm] (one) to (nine);

\draw[scochainm] (one) to node []{$ $} (three);
\draw[scochainm] (nine) to node []{$ $} (seven);

\draw[scochainm] (four.-60+\bend+\dangle) to [bend left=\bend](twelve.120-\bend-\dangle);
\draw[scochainm] (ten) to [bend left=\bend](six);

\draw[scochainm]  (twelve.150-\dangle) to (seven.-30+\dangle);
\draw[scochainm] (one) to (six);
\draw[scochainm] (three) to (ten);
\draw[scochainm] (four) to (nine);

\draw[scochainm] (four.-90+\dangle) to (ten.90-\dangle);
\draw[scochainm] (three) to (nine);
\draw[scochainm] (one.180-\dangle) to (seven.0+\dangle);
\draw[scochainm] (twelve.150-\dangle) to (six.-30+\dangle);
\end{tikzpicture}
\end{tabular}
\end{center}
\end{table}
\end{center}

Summarising all this we have the following explicit description of $\aliak[{\mg}]{0}$ by generators and relations.
\begin{Theorem}
\label{thm:aliam}
Let $\rho$ be a restricted representation with $\rho(-\Id)=\id$. Then the Lie algebra $\aliak[{\mg}]{0}$ is generated as a free $\mb{C}[j]$-module by elements $h_i(\tau)$ and $\ba_\alpha(\tau)$ where $i=1\ldots,N$. The $\mb{C}[j]$-linear Lie structure is given by the brackets
\[
\begin{array}{rll}
{[}h_i, h_j]&\backspace=0&
\\{[}h_i, \ba_\alpha]&\backspace=\alpha(H_i) \ba_\alpha&
\\{[}\ba_\alpha, \ba_{-\alpha}]&\backspace=j^{\omega^2_4(\alpha,-\alpha)}(j-1728)^{\omega^2_6(\alpha,-\alpha)}\sum n_ih_i,& \alpha=\sum n_i\alpha_i,
\\{[}\ba_\alpha,\ba_\beta]&\backspace=\epsilon(\alpha,\beta)j^{\omega^2_4(\alpha,\beta)}(j-1728)^{\omega^2_6(\alpha,\beta)} \ba_{\alpha+\beta},& \alpha+\beta\in\roots,
\\{[}\ba_\alpha,\ba_\beta]&\backspace=0,& \alpha+\beta\notin\roots\cup\{0\},
\end{array} 
\]
where $\om$ and $\omm$ are the symmetric $2$-cocycles (\ref{eq:2cocycles}) on the root system $\roots$ taking only values $0$ and $1$. 
\end{Theorem}

\section{Principal Congruence Subgroups and Polyhedral Loop Algebras}
\label{sec:principal}

Let now $\Gamma=\Gamma(N)$ be the principal congruence subgroup of modular group
$$\Gamma(N)=\left\{\left( \begin{array}{cc}
a&b\\
c&d
\end{array}\right) \in \SLNZ[2]: \left( \begin{array}{cc}
a&b\\
c&d
\end{array}\right)\equiv \left( \begin{array}{cc}
1&0\\
0&1
\end{array}\right) \mod{N}\right\}.
$$
Let $X(N)$ be the compactification of the corresponding quotient $Y(N):=\Gamma(N)\backslash\uh$ by adding the parabolic points with appropriate local coordinates (see Ch. 2 in \cite{gunning1962lectures}). The genus of $\Gamma(N)$ is by definition the genus of $X(N)$.
 
We start with the following remarkable relation between the genus zero principal congruence subgroups and regular polyhedra, which probably should be ascribed to Klein (see \cite{klein1966vorlesungenI,klein1966vorlesungenII,gunning1962lectures}). 

\begin{Theorem} 
The modular curve $X(N)$ has genus zero if and only if $1\le N\le 5$. An isomorphism from $X(N)$ to the Riemann sphere can be chosen such that the cusps are mapped to the vertices of a regular triangle, tetrahedron, octahedron and icosahedron for $N=2,3,4,5$ respectively. 

The quotient group $\Gamma(1)/\Gamma(N)$ is isomorphic to dihedral group $\mathcal D_3\cong S_3$ (for $N=2$) and to the groups $\mathcal T \cong A_4$, $\mathcal O \cong S_4$, $\mathcal I \cong A_5$ of the rotational symmetries of the corresponding regular polyhedron (for $N=3,4,5$).
\end{Theorem}

\newcommand{\np}{M}
We consider now the automorphic Lie algebra $\aliak[{\Gamma(N)}]{0}$ in more detail in the special cases of the principal congruence subgroups $\Gamma(N), \, 2\le N\leq 5$, when the corresponding quotients $X(\rg)$ have genus zero. Remarkably, these cases are closely related to the platonic solids and the corresponding $\np$-point loop algebras \[\mf{g}\otimes_\mb{C} \mb{C}[t,(t-a_1)^{-1},\ldots,(t-a_{\np-1})^{-1}]\] of $\mf{g}$-valued meromorphic functions on the Riemann sphere with poles allowed only at $\np$ points $S=\{\infty=a_0,a_1,\ldots,a_{\np-1}\}.$

Their central extensions were studied by Bremner \cite{bremner1994generalized,bremner1995four}, who was inspired by the important work of Krichever and Novikov \cite{krichever1987algebras}
and Schlichenmaier \cite{schlichenmaier1990krichever} 
on the Lie algebra of vector fields on Riemann surfaces.
\newcommand{\Ta}{\mf{Tg}}
\newcommand{\Oa}{\mf{Og}}
\newcommand{\Ia}{\mf{Ig}}
\newcommand{\Da}{\mf{Dg}}

We will show that the three particular cases of such algebras with $S$ being the sets of vertices of the regular tetrahedron, octahedron and icosahedron respectively, naturally appear as automorphic Lie algebras related to principal congruence subgroups $\Gamma(3), \Gamma(4), \Gamma(5).$ The corresponding sets of vertices can be chosen as
$$
S_T=\{\infty, 1, e^{\pm\frac{2\pi i}{3}} \}, \quad S_O=\{0, \infty, \pm 1, \pm i\},
$$
$$
S_I=\{0, \infty, \varepsilon^j(\varepsilon+\varepsilon^4), \varepsilon^j(\varepsilon^2+\varepsilon^3)\}, \quad j=0,\dots, 4, 
$$
where $\varepsilon=e^{\frac{2\pi i}{5}}$ (see \cite{klein1993vorlesungen}).

We call the corresponding Lie algebras 
\beq{tet}
\Ta=\mf{g}\otimes_\mb{C} \mb{C}[t, (t-1)^{-1},(t-\omega)^{-1},(t-\bar\omega)^{-1}], \,\, \omega=e^{\frac{2\pi i}{3}},
\eeq
\beq{oct}
\Oa=\mf{g}\otimes_{\mb{C}}\mb{C}[t,t^{-1},(t-1)^{-1},(t+1)^{-1},(t-i)^{-1},(t+i)^{-1}],
\eeq
\beq{icos}
\Ia=\mf{g}\otimes_{\mb{C}}\mb{C}[t,t^{-1},(t-\varepsilon^j(\varepsilon+\varepsilon^4))^{-1},(t-\varepsilon^j(\varepsilon^2+\varepsilon^3))^{-1}: j=0,\dots,4]
\eeq
{\it tetrahedral, octahedral and icosahedral loop algebras} respectively.

The principal congruence subgroup $\Gamma(2)$ is related to the {\it dihedron case} (in Klein's terminology \cite{klein1993vorlesungen}
) and to the 3-point loop algebra 
\beq{dih}
\Da=\mf{g}\otimes_\mb{C} \mb{C}[t,t^{-1},(t-1)^{-1}].
\eeq
Its particular case
\[\mf{sl}(2,\mb{C})\otimes_\mb{C} \mb{C}[t,t^{-1},(t-1)^{-1}]\]
was recently studied by Hartwig and Terwilliger \cite{hartwig2007tetrahedron}, who called it the tetrahedron algebra for different reasons. Note that all 3-point algebras are isomorphic since any triples of points on Riemann sphere are $\SLNC[2]$-equivalent.

From our results in Section \ref{sec:zero} we have the following 
\begin{Theorem} 
\label{thm:alias polyhedral}
For the principal congruence subgroups $\Gamma(N), \, N=3,4,5$ and their restricted representations, the zero weight automorphic Lie algebras $\alia$ are isomorphic to tetrahedral, octahedral and icosahedral loop algebras (\ref{tet}),(\ref{oct}),(\ref{icos}) respectively.

For $\Gamma(2)$ with even representation $\rho(-\Id)=\id$, the automorphic Lie algebra $\alia$ is isomorphic to dihedral loop algebra (\ref{dih}).
\end{Theorem}

In the rest of the section we present more details about this isomorphisms using the explicit description of the corresponding modular forms going back to Klein and Fricke \cite{klein1966vorlesungenI,klein1966vorlesungenII}.

\subsection{Theta series and modular forms}
\label{sec:theta}

We will need now the following classical way to produce some modular forms going back to Jacobi, where we use the notations of Whittaker and Watson \cite{whittaker1996a} with $$q=e^{\pi i \tau},$$ 
rather than $e^{2\pi i \tau}$ as used before.
The theta series (or, theta constants) are given by
\newcommand{\jj}{\theta_2}
\newcommand{\jjj}{\theta_3}
\newcommand{\jjjj}{\theta_4}
\begin{align*}
\jj(\tau)&=\sum_{n\in\mb{Z}}q^{(n+1/2)^2}=2q^{1/4}+2q^{9/4}+2q^{25/4}+2q^{49/4}\ldots,
\\\jjj(\tau)&=\sum_{n\in\mb{Z}}q^{n^2}=1+2 q+2 q^{4}+2 q^{9}+\ldots,
\\\jjjj(\tau)&=\sum_{n\in\mb{Z}}(-1)^nq^{n^2}=1-2 q+2 q^{4}-2 q^{9}+\ldots.
\end{align*}
They are the values at $z=0$ of corresponding theta functions $\theta_k(z,\tau)$ introduced by Jacobi.

We will recall a few facts about theta series here, which can be found in \cite{mumford1983tata} (where another common notation is used: $\theta_{10}=\jj$, $\theta_{00}=\jjj$ and $\theta_{01}=\jjjj$). 

The theta series satisfy Jacobi's identity
\beq{eq:jacobi}
\jj^4+\jjjj^4=\jjj^4.
\eeq
Transforming the parameter by the modular group gives
\begin{align}
\label{eq:T on theta}
\begin{pmatrix}
\jj\\\jjj\\\jjjj
\end{pmatrix}\!\!(\tau+1)
&=
\begin{pmatrix}
\zeta&0&0\\
0&0&1\\
0&1&0
\end{pmatrix}
\begin{pmatrix}
\jj\\\jjj\\\jjjj
\end{pmatrix}\!\!(\tau)
\\
\label{eq:S on theta}
\begin{pmatrix}
\jj\\\jjj\\\jjjj
\end{pmatrix}\!\!(-1/\tau)
&=
\tau^{1/2}\zeta^{-1}
\begin{pmatrix}
0&0&1\\
0&1&0\\
1&0&0
\end{pmatrix}
\begin{pmatrix}
\jj\\\jjj\\\jjjj
\end{pmatrix}\!\!(\tau)
\end{align}
where $\zeta=\exp{2\pi i/8}$.
Gannon \cite{gannon2014the} gives an interpretation of this vector as a VVMF in a more general definition than presented in this paper.

The product $\jj^8\jjj^8\jjjj^8$ is a $\SLNZ[2]$-modular form of weight $12$ and therefore a constant multiple of $\Delta=q\prod_{n\ge 1} (1-q^n)^{24}$. The theta series are in particular never zero in $\uh$ and at least one of them is zero at any cusp.

\subsection{$\Gamma(2)$ and theta constants}
\label{subsec:gamma2}

The following result is well-known in the theory of modular forms (see e.g. \cite{rankin1977modular}).
\begin{Theorem}
\label{lem:Gamma(2) modular forms}
The ring of $\Gamma(2)$-modular forms is generated by $\jj^4,\jjj^4,\jjjj^4$
with the only relation being $\jj^4+\jjjj^4=\jjj^4$. In particular, $$\mfpz[\Gamma(2)]\cong \mb{C}[x,y].$$
\end{Theorem}

The $\Gamma(2)$-modular forms can also be constructed starting from the Eisenstein series. One can check that
$$
F_2(\tau)=2E_2(2\tau)-E_2(\tau), \quad H_2(\tau)=F_2(\tau/2)
$$
are $\Gamma(2)$-modular forms of weight $2$.
If we recall the expansion
\[E_2=1 - 24q^{2} - 72q^{4} - 96q^{6} - 168q^{8} - 144q^{10} + \ldots,\quad q=\exp(\pi i \tau),\]
we can compute
\begin{align*}
F_2&=1 + 24q^{2} + 24q^{4} + \ldots
\\H_2&=1 + 24q + 24q^{2} + 96q^{3} + 24q^{4} + 144q^{5}+ 
 \ldots.
\end{align*}
Taking linear combinations of $F_2$ and $H_2$ to constructing a cusp form for each of the three cusps of $\Gamma(2)$ leads to the fourth powers of the theta series:
\begin{align*}
\jj^4&=-\frac{2}{3}F_2+\frac{2}{3}H_2,
\\\jjj^4&=\frac{2}{3}F_2+\frac{1}{3}H_2,
\\\jjjj^4&=\frac{4}{3}F_2-\frac{1}{3}H_2.
\end{align*}

From (\ref{eq:T on theta}) and (\ref{eq:S on theta}) we see that
\begin{align*}
\begin{pmatrix}
\jj^4(\tau+1)\\\jjj^4(\tau+1)\\\jjjj^4(\tau+1)
\end{pmatrix}
&=
\begin{pmatrix}
-1&0&0\\
0&0&1\\
0&1&0
\end{pmatrix}
\begin{pmatrix}
\jj^4(\tau)\\\jjj^4(\tau)\\\jjjj^4(\tau)
\end{pmatrix}
\\
\begin{pmatrix}
\jj^4(-1/\tau)\\\jjj^4(-1/\tau)\\\jjjj^4(-1/\tau)
\end{pmatrix}
&=
-\tau^{2}
\begin{pmatrix}
0&0&1\\
0&1&0\\
1&0&0
\end{pmatrix}
\begin{pmatrix}
\jj^4(\tau)\\\jjj^4(\tau)\\\jjjj^4(\tau)
\end{pmatrix}
\end{align*}
so that $(\jj^4,\jjj^4,\jjjj^4)$ is a VVMF for the full modular group of weight $2$. The image of the representation is isomorphic to $S_3\cong \SLNZ[2]/\Gamma(2)$.

Let us introduce the classical $\lambda$-invariant (see \cite[Section 7.2]{rankin1977modular}) as
\[\lambda=\frac{\jj^4}{\jjj^4}.\]
This meromorphic function has no zeros or poles in $\uh$ since the theta functions have no zeros in $\uh$. Its meromorphic extension to $\overline{\uh}$ defines an isomorphism of Riemann surfaces $\Gamma(2)\backslash\overline{\uh}\to\overline{\mb{C}}$. This isomorphism sends the cusps to 
\[{[}\infty]\mapsto 0,\quad [0]\mapsto 1,\quad [-1]\mapsto\infty.\] 
It follows that $\lambda$ is a Hauptmodul for $\Gamma(2)$, in the sense that every  $\Gamma(2)$-invariant meromorphic function on the upper half plane is a rational function in $\lambda$. In particular, the Hauptmodul $j$ for $\Gamma(1)$ is a rational function in $\lambda$, and in fact 
\[j=256\frac{(\lambda^2-\lambda+1)^3}{\lambda^2(\lambda-1)^2}.\] 
The $\lambda$-invariant appears also in the parametrisation of the elliptic curves in the form
$$
y^2=x(x-1)(x-\lambda), \quad \lambda \neq 0,1.
$$

Our interest is in the algebra $\mfk[\Gamma(2)]{0}$ of $\Gamma(2)$-invariant meromorphic functions on $\overline{\uh}$, which are holomorphic on $\uh$. By the preceding paragraph, these are precisely the rational functions in $\lambda$ which are only allowed poles at $\lambda\in\{\infty,0,1\}$, that is
\beq{eq:Gamma(2) modular functions}
\mfk[\Gamma(2)]{0}=\mb{C}\left[\lambda,\lambda^{-1},(\lambda-1)^{-1} \right].
\eeq
The action of the modular group
$$T\lambda=\frac{\lambda}{\lambda-1},\quad S\lambda=1-\lambda,$$
deduced from (\ref{eq:T on theta},\ref{eq:S on theta},\ref{eq:jacobi}) realises the full group of M\H{o}bius transformations preserving the set $\{\infty,0,1\}$.

As a special case of Theorem \ref{thm:alias even group}, using that $\theta_i^4$ is a nonvanishing modular form of weight $2$, we have 
\begin{Theorem}
\label{thm:Gamma(2) alias}
The algebra $\alia[\Gamma(2)]$ with $\rho(-\Id)=\id$ is isomorphic to the dihedral loop algebra
\[\alia[\Gamma(2)]\cong \mf{g}\otimes_{\mathbb C} \mfk[\Gamma(2)]{0}\]
as Lie algebra and module over $\mfk[\Gamma(2)]{0}=\mb{C}\left[\lambda,\lambda^{-1},(\lambda-1)^{-1} \right]$.
\end{Theorem}

To conclude we briefly describe interesting results of Hartwig and Terwilliger \cite{hartwig2007tetrahedron} about their {\it tetrahedron algebra}, which is isomorphic to the $\mf{sl}(2,\mb{C})$ case of our dihedral loop algebra.

Their tetrahedron algebra $\boxtimes$ is defined as the Lie algebra over a field $\mb{K}$ with generators
\[X_{i,j},\quad i,j=1,2,3,4,\quad i\ne j\]
satisfying the relations
\begin{enumerate}
\item $X_{ij}+X_{ji}=0$ for distinct $i,j$,
\item $[X_{hi},X_{ij}]=2X_{hi}+2X_{ij}$ for mutually distinct $h,i,j$,
\item $[X_{hi},[X_{hi},[X_{hi},X_{jk}]]]=4[X_{hi},X_{jk}]$ for mutually distinct $h,i,j,k$.
\end{enumerate}
It has the following interesting properties \cite{hartwig2007tetrahedron}.
\begin{Theorem} [Hartwig and Terwilliger]
Let us  label the vertices of a tetrahedron by the indices $1,2,3,4$ and identify each edge $\{i,j\}$ of the tetrahedron with a one-dimensional subspace $\mb{K}X_{ij}$ of $\boxtimes$. Then
\begin{enumerate}
\item the boundary of one face of the tetrahedron spans a subalgebra of $\boxtimes$ isomorphic to $\mf{sl}(2,\mb{K})$,
\item two edges that don't share a vertex generate a subalgebra of $\boxtimes$ isomorphic to the Onsager algebra,
\item five out of six edges generate a subalgebra of $\boxtimes$ isomorphic to the loop algebra  $\mf{sl}(2,\mb{K})\otimes_\mb{K} \mb{K}[t,t^{-1}]$,
\item there is an isomorphism \[\boxtimes\cong \mf{sl}(2,\mb{K})\otimes_\mb{K} \mb{K}[t,t^{-1},(t-1)^{-1}].\]
\item Let $\mf{O}$, $\mf{O}'$ and $\mf{O}''$ be the subalgebras generated by the three pairs of opposite edges (each isomorphic to the Onsager algebra due to the second item). Then there is a direct sum $\boxtimes=\mf{O}\oplus\mf{O}'\oplus\mf{O}''$ of $\mb{K}$-spaces.
\end{enumerate}
\end{Theorem}

\subsection{$\Gamma(3)$ and the Hesse cubic}
\label{subsec:gamma3}

The principal congruence subgroup $\Gamma(3)$ is related to the Hesse form of the elliptic curves given in projective coordinates as cubics
\beq{Hesse}
x^3+y^3+z^3+\gamma xyz=0, \quad \gamma^3+27 \neq 0.
\eeq
The Hesse parameter $\gamma$ is related to $j$-invariant by the formula
\beq{jg}
j=\frac{\gamma^3(216-\gamma^3)^3}{(27+\gamma^3)^3}
\eeq
and is a Hauptmodul of $\Gamma(3)$ (see \cite{dolgachev2005lectures}).

The corresponding quotient $Y(3)=\Gamma(3)\backslash\uh$ is a sphere punctured at 4 points: infinity and 3 cube roots of unity, with an isomorphism established by the function $-\gamma/3.$ This implies that
$$
\mfk[\Gamma(3)]{0}\cong \mathbb C[z, (z-1)^{-1}, (z-\varepsilon)^{-1}, (z-\bar\varepsilon)^{-1}], \quad \varepsilon =e^{\frac{2\pi i}{3}}
$$
and that the corresponding automorphic Lie algebra $\aliaz$ is isomorphic to the tetrahedral loop algebra.

Note that the tetrahedral loop algebra is a particular case of 4-point loop algebras
studied by Bremner \cite{bremner1995four}. 

The ring of modular forms 
$
\mfpz[\Gamma(3)]=\mathbb C[\varphi_1, \,\varphi_2]
$
of $\Gamma(3)$ is generated by two weight 1 forms
\beq{m3}
\varphi_1=\sum_{(x,y)\in \mathbb Z^2}q^{x^2-xy+y^2}, \quad \varphi_2=q^{1/3}\sum_{(x,y)\in \mathbb Z^2}q^{x^2-xy+y^2+x-y},
\eeq
appeared in connection with the weight enumerators of ternary self-dual codes (see \cite{ebeling2013lattices, bannai2001some}). 
They are related to the Eisenstein series by the formulae
\beq{rel3}
E_4=u^4+8uv^3, \,\, E_6=u^6-20u^3v^3-8v^6, \quad u=\varphi_1, \, v=\varphi_2.
\eeq
The modular form $\varphi_1$ is known also as the Eisenstein form $E_{1,3}$ and has appeared in the classification of integrable Lagrangian partial differential equations (PDEs) done by Ferapontov and Odesski \cite{ferapontov2010integrable}. More precisely, they have shown that it satisfies the following nonlinear fourth order ordinary differential equation (ODE)
$$
g''''(g^2g''-2g(g')^2)-9(g')^2(g'')^2+2g g'g''g'''+8(g')^3g'''-g^2(g''')^2=0,
$$
which is the integrability condition of PDEs with Lagrangian of the form $f=u_xu_yg(u_t)$ (see section 3 in \cite{ferapontov2010integrable}).

\subsection{$\Gamma(4)$ and the octahedron} 
\label{subsec:gamma4}
We will follow here mainly the excellent presentation of
$\Gamma(4)$-modular forms by Mumford in Chapter 1 of his book \cite{mumford1983tata}. 

Consider first the set $S=\{\infty,0,\pm 1, \pm i\} \subset \mathbb CP^1\cong S^2$, which we regard as the vertices of an octahedron.
We will see that the quotient $$\Gamma(4)\backslash\uh\cong \mathbb CP^1\setminus S$$ is the complex sphere punctured at this set.

One can check that the group of holomorphic automorphisms of $\overline{\mb{C}}$ preserving this set is indeed the octahedral group generated by
\beq{symoc}
\mc{T}:t\mapsto \frac{\zeta^3 t+\zeta}{\zeta t+\zeta^3},\quad
\mc{S}:t\mapsto\frac{-t+1}{t+1}.
\eeq
where $\zeta=\exp(2\pi i/8)$. These two M\"obius transformation satisfy $$\mc{T}^4=\id,\quad\mc{S}^2=\id,\quad(\mc{T}\mc{S})^3=\id.$$

\begin{Theorem} [\cite{mumford1983tata}]
\label{thm:Gamma(4) modular curve}
The compactification $X(4)$ of modular curve  is isomorphic as a Riemann surface to $A \cong \mb{C}P^1 \cong \overline{\mb{C}}$. Isomorphisms are realised by the map
$\Gamma(4)\backslash\overline{\uh}\to A=\{[x_0:x_1:x_2]\in\mb{C}P^2\,|\,x_0^2=x_1^2+x_2^2\}$ given by
\[\tau\mapsto [\jjj^2(\tau):\jjjj^2(\tau):\jj^2(\tau)],\]
the map $A\to\mb{C}P^1$ given by
\[x_0\mapsto t_0^2+t_1^2,\quad x_1\mapsto 2t_0t_1,\quad x_2\mapsto t_0^2-t_1^2,\]
where $t_0,t_1$ are homogeneous coordinates of $\mb{C}P^1$,
and the map $\mb{C}P^1\to\overline{\mb{C}}$ given by 
$$[t_0:t_1]\mapsto\left\{\begin{array}{ll} t=t_1/t_0&\text{if } t_0\ne 0,\\ \infty&\text{if }t_0=0.\end{array}\right.$$ 
The cusps land in the following points.
$$\begin{array}{cccc}
\Gamma(4)\backslash\overline{\uh}&A&\mb{C}P^1&\overline{\mb{C}}
\\\hline
{[}\infty]&[1:1:0]&[1:1]&1
\\{[}0]&[1:0:1]&[1:0]&\infty
\\{[}1/2]&[1:-1:0]&[1:-1]&-1
\\{[}1]&[0:1:i]&[i:1]&i
\\{[}2]&[1:0:-1]&[0:1]&0
\\{[}3]&[0:1:-i]&[-i:1]&-i
\end{array}$$
Restriction to the upper half plane establishes an isomorphism $\Gamma(4)\backslash\uh\cong\overline{\mb{C}}\setminus\{0,\infty,\pm1, \pm i\}.$
\end{Theorem}
This theorem shows that 
\beq{hmg4}
\mu=\frac{\jjjj^2}{\jj^2+\jjj^3}
\eeq
is a Hauptmodul for $\Gamma(4)$.
The algebra $\mfk[\Gamma(4)]{0}$ consists of the rational functions in $\mu$ which are only allowed poles at $\mu\in\{\infty,0,\pm1,\pm i\}$, that is
\beq{eq:Gamma(4) modular functions}
\mfk[\Gamma(4)]{0}=\mb{C}\left[\mu,\mu^{-1},(\mu-1)^{-1},(\mu+1)^{-1},(\mu-i)^{-1},(\mu+i)^{-1}\right].
\eeq

One can check also that the action of the modular group is given by formulae (\ref{symoc}) and thus is reduced to the action of the octahedral group.

The structure of the ring of modular forms can be deduced from Theorem \ref{thm:Gamma(4) modular curve}, see \cite[Corollary 10.2]{mumford1983tata}.
\begin{proposition}
\label{lem:Gamma(4) modular forms}
The ring of $\Gamma(4)$-modular forms is generated by $\jj^2,\jjj^2$ and $\jjjj^2$ 
with the only relation being $\jj^4+\jjjj^4=\jjj^4$. In particular, $$\mfpz[\Gamma(4)]\cong\mb{C}[x_0,x_1,x_2]/(x_0^2-x_1^2-x_2^2).$$
\end{proposition}

As a corollary and special case of Theorem \ref{thm:alias odd group} we have
\begin{Theorem}
\label{thm:Gamma(4) alias}
The algebra $\alia[\Gamma(4)]$ is isomorphic to the octahedral loop algebra
\[\alia[\Gamma(4)]\cong \mf{g}\otimes_{\mathbb C} \mfk[\Gamma(4)]{0}\]
as Lie algebra and module over $$\mfk[\Gamma(4)]{0}=\mb{C}\left[\mu,\mu^{-1},(\mu-1)^{-1},(\mu+1)^{-1},(\mu-i)^{-1},(\mu+i)^{-1}\right].$$
\end{Theorem}

\subsection{$\Gamma(5)$ and Klein's modular forms}
\label{subsec:gamma5}

We follow Schultz's lectures \cite{schultz2015notes} to present an explicit description of $\Gamma(5)$ modular forms going back to Klein and Fricke \cite{klein1966vorlesungenI,klein1966vorlesungenII}.

Recall the Klein forms (\ref{Klein}).
In \cite[Proposition 4.9.6]{schultz2015notes} we find
\begin{Theorem}
The weight $k$ modular forms of $\Gamma(5)$ can be expressed explicitly in terms of Klein's forms as follows
\beq{Klein2}
M_k(\Gamma(5))=\oplus_{a+b=5k,\, a,b\geq 0} \mathbb C \frac{\eta(5\tau)^{15k}}{\eta(\tau)^{3k}}\mf{k}_{\frac{1}{5},\frac{0}{5}}(5\tau)^a\mf{k}_{\frac{2}{5},\frac{0}{5}}(5\tau)^b,
\eeq
where $\eta(\tau)$ is the Dedekind eta-function (\ref{Dedekind}).
In particular, 
\beq{f}
f=\frac{\eta(5\tau)^{15}\mf{k}_{\frac{1}{5},\frac{0}{5}}(5\tau)^5}{\eta(\tau)^3}
\eeq
belongs to $\mfk[\Gamma(5)]{-1}$, has a pole of order 5 at infinity and does not vanish anywhere in $\uh$.
\end{Theorem}

The modular curve $Y(5)=\Gamma(5)\backslash\uh$ is a complex sphere punctured at the vertices of the icosahedron \cite{gunning1962lectures}, so as a corollary we have the isomorphism of the corresponding $\alia$ with the icosahedral loop algebra.

\section{Genus One Case: $\Gamma(6)$ and Markov Triples}
\label{sec:gamma6}
Let us consider now the subgroups $\Gamma$ of the modular group such that $X(\Gamma)$ is a smooth curve of genus 1.
In particular, the congruence subgroups 
$$
\Gamma_0(N):=\left\{\left( \begin{array}{cc}
a&b\\
c&d
\end{array}\right) \in \SLNZ[2]: c \equiv 0 \mod{N}\right\}
$$
with $$N=11, 14, 15, 17, 19, 20, 21, 27, 32, 36, 49$$ have this property.

Since any elliptic curve has holomorphic form $dz^{1/2}$ of weight $1$ without zeros, our Theorem \ref{thm:alias odd group} implies
\begin{Theorem}
\label{thm:alias genus 1 groups}
For any genus 1 finite index subgroup $\Gamma$ and its restricted representation $\rho$ we have a Lie algebra isomorphism
\[\alia[\Gamma]\cong \mf{g}\otimes_{\mathbb C} \mfk[\Gamma]{0}.\]
\end{Theorem}
In particular, this is true for $\Gamma(6)$, which is the only principal congruence subgroup with genus 1 \cite{gunning1962lectures}. We consider now this case in more detail.

The corresponding quotient $Y(6)=\Gamma(6)\backslash\uh$ has $12$ cusps and topologically is a torus with 12 punctures \cite{gunning1962lectures}.
The group $$P\Gamma(1)/P\Gamma(6)\cong S_3 \times A_4$$ acting on $Y(6)$ has order $72$ \cite{rankin1977modular}. The orbit of cusps therefore has stabiliser of order $6$. There is only one complex torus (known as {\it equianharmonic}) which has  such a symmetry, which is $$X(6)\cong \mb{C}/\mb{Z}\oplus\mb{Z}e^{2\pi i/3}.$$ In the Weierstrass form 
the corresponding elliptic curve corresponds to the case with $g_2=0.$ 

The corresponding set of punctures shown on Fig. 1 is an orbit of the group of translations $\mathbb Z_2 \times \mathbb Z_6$ generated by
$$z \to z+\frac{1}{2}, \,\, \, z \to z+\frac{1}{3}+\frac{1}{6}e^{2\pi i/3},$$ so that $$P\Gamma(1)/P\Gamma(6)\cong \mathbb Z_6\ltimes (\mathbb Z_2 \times \mathbb Z_6)$$ can be represented also as a semi-direct product  of this group with $\mathbb Z_6$ generated by $z\to e^{\pi i/3} z$.

\begin{figure}[h]
\centering
  \includegraphics[width=70mm]{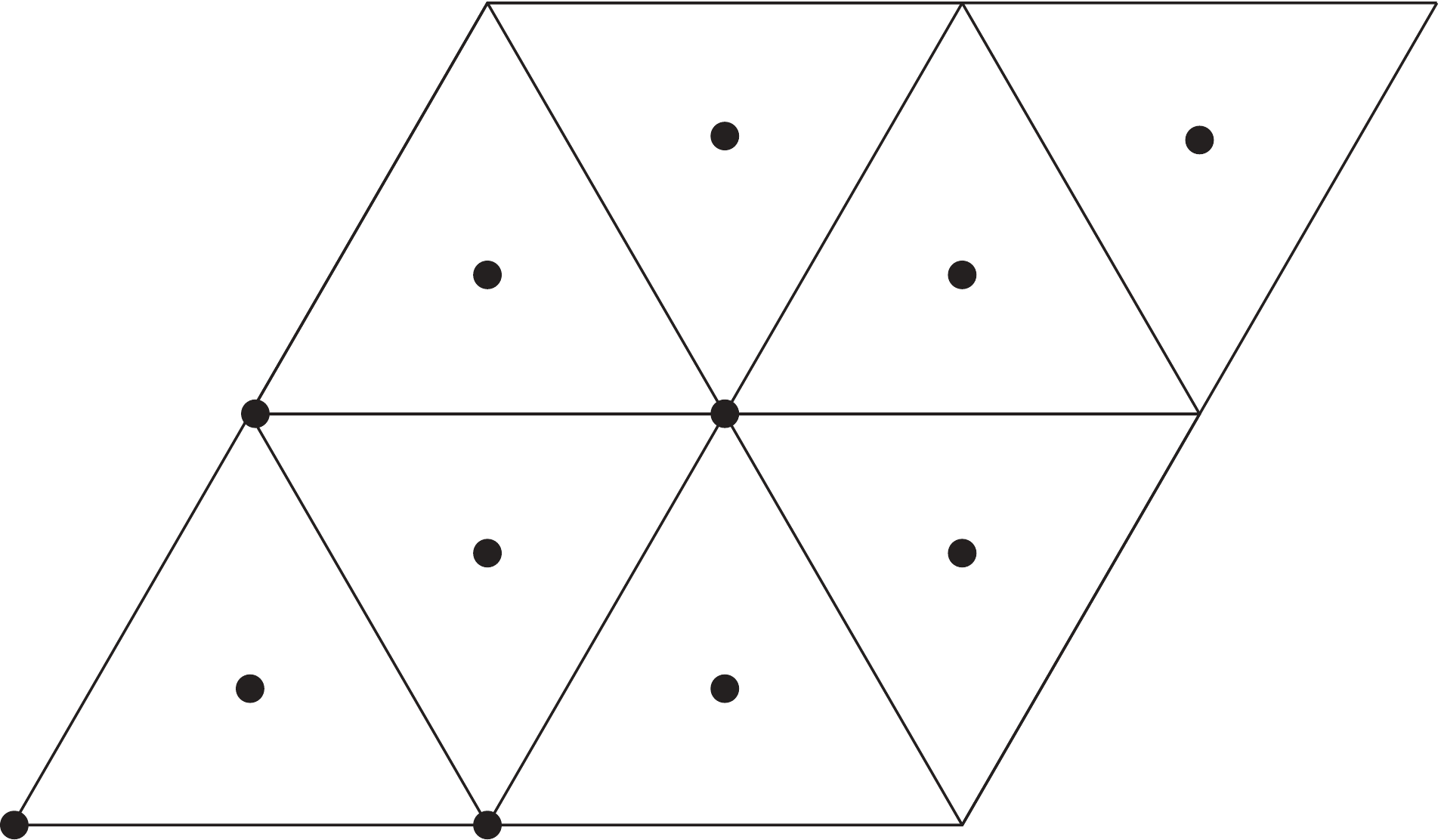}  
  \caption{12 punctures of $Y(6)$}
\end{figure}

There is another important closely related congruence subgroup, namely the commutator of the modular group $\Gamma'=\Gamma(1)'$.
It is known  that $\Gamma'$ is a free group generated by the matrices
$$
U=\begin{pmatrix}
2&1\\
1&1
\end{pmatrix}, \quad V=\begin{pmatrix}
1&1\\
1&2
\end{pmatrix}
$$
and contains $\Gamma(6)$ as a subgroup of index 12 \cite{newman1962the, rankin1977modular}.

The corresponding quotient $Y(\Gamma')=\Gamma'\backslash\uh$ is a one-punctured equianharmonic torus, which is 12-covered by $Y(6)$:
$$
Y(\Gamma')=G\backslash Y(6), \quad G=\Gamma'/\Gamma(6).
$$
This particular one-punctured torus turned out to be closely related to the celebrated Markov triples satisfying the Diophantine equation
$$
x^2+y^2+z^2=3xyz.
$$
There is a natural $PGL_2(\mathbb Z)$-action on the solutions of this equation generated by permutation group $S_3$ and Vieta involution
$$
(x,y,z) \to (x,y, 3xy-z).
$$
Markov \cite{markoff1880sur} showed that set of all positive integer solutions of this equation is just one $PGL_2(\mathbb Z)$-orbit of the solution $(1,1,1):$
$$
(1,1,1), (1,1,2), (1,2,5), (1,5,13), (2,5,29), \dots.
$$
The elements of Markov triples are known as {\it Markov numbers}:
$$
1, 2, 5, 13, 29, 34, 89, 169, 194, 233, 433, 610, 985, \dots.
$$
A conjecture due to Frobenius (see Aigner \cite{aigner2013markov}) claims that the maximal element $z$ of Markov triple $(x,y,z), \, x\le y \le z$ uniquely determines the triple, so there are as many Markov triples as Markov numbers.

Markov triples originally appeared in number theory, but in the 1950s Gorshkov and Cohn \cite{cohn1955approach, gorshkov1981geometry} independently discovered the following remarkable relation with the hyperbolic geometry of the Markov torus (see more details in \cite{aigner2013markov, haas1986diophantine, series1985the} and for the most recent proof \cite{springborn2017the}).

\begin{Theorem} (Gorshkov, Cohn)
Markov numbers can be interpreted as $\frac{2}{3} \cosh l$ of the lengths $l$ of the corresponding simple closed geodesics on Markov one-punctured torus.
\end{Theorem}

More recently Markov triples appeared in many other interesting links with algebraic geometry, theory of Frobenius manifolds and quantum cohomology, see the references in the review \cite{bukhshtaber2019conway}.

As a corollary to Theorem \ref{thm:alias genus 1 groups}, we have
\begin{Theorem}
\label{thm:alias Markov}
For the commutator $\Gamma'$ of the modular group we have a Lie algebra isomorphism
\[\alia[\Gamma']\cong \mf{g}\otimes_{\mathbb C} \mfk[\Gamma']{0}.\]
The algebra $\mfk[\Gamma']{0}\cong\mathbb C[x,y]$ with $x,y$ satisfying the relation
$$
y^2=4x^3-1,
$$
where $x=\wp(z), \,\ y=\wp'(z)$ are the Weierstrass functions on the corresponding equianharmonic elliptic curve.
\end{Theorem}

Since $P(\Gamma(1)/\Gamma')$ is the cyclic group of order 6, $Y(\Gamma')$ is a 6 times cover of the modular curve $\Gamma(1)\backslash\uh$. The holomorphic form $dz^6$ on $Y(\Gamma')$ gives rise to the classical weight 12 modular form $$\Delta=\frac{E_4^3-E_6^2}{1728}=q\prod_{n=1}^\infty(1-q^n)^{24},$$ which provides a nice geometric description of this form.

Note that $\Gamma'$ is a congruence subgroup (since it contains $\Gamma(6)$), but it is not a principal congruence subgroup of the modular group.

The list of all the subgroups or $PSL(2,\mathbb Z)$ of index up to 7 can be found in Schultz's lectures \cite{schultz2015notes}. Apart from one exception, all of them have genus 0 with the Hauptmodul given by an explicit algebraic relation with the absolute invariant $j$ (see Table 4.1 in \cite{schultz2015notes}). The corresponding automorphic Lie algebras are worthy of studying in more details.

\section{Extensions and Representations}
\label{sec:extensions}

Recall that the loop algebra $\mathcal L(\mathfrak g)=\mathfrak g\otimes_{\mathbb C}\mathbb C[z, z^{-1}] $ of a simple Lie algebra $\mathfrak g$ has the following central extension
$\hat{\mathcal L}(\mathfrak g)=\mathcal L(\mathfrak g) \oplus \mathbb C c$
defined by the 2-cocycle
\beq{cocy}
\omega(f(z)x, g(z)y)=\frac{1}{2\pi i}  \oint_C f dg  \,K(x,y), \quad f,g \in \mathbb C[z, z^{-1}], \, x,y \in \mathfrak g,
\eeq
where $C$ is a contour surrounding $z=0$ and $$K(x,y):= \tr \, \ad x\,  \ad y$$ is the Killing form on $\mathfrak g:$
$$
[x z^m, y z^n]:=[x,y] z^{m+n}+mK(x,y)\delta_{m+n,0} c.
$$ 
It is known that the central extensions of Lie algebra $\mathcal L(\mathfrak g)$ are classified by the second cohomology group $H^2(\mathcal L(\mathfrak g), \mathbb C).$
The above extension is universal since this cocycle generates $H^2(\mathcal L(\mathfrak g), \mathbb C)$, which is one-dimensional (see e.g. \cite{kac1990infinite}).

The universal central extensions of its natural generalisations 
$$
\mathcal L_\np(\mathfrak g)=\mathfrak g\otimes_{\mathbb C}\mathbb C[z, (z-a_1)^{-1}, \, (z-a_{\np-1})^{-1}]
$$
called $\np$-point loop algebras, were studied by Bremner \cite{bremner1994universal}, 
who proved that the corresponding $H^2(\mathcal L_\np, \mathbb C)$ is ($\np$-1)-dimensional with a basis given by the cocycles
\beq{cocy2}
\omega_k(f(z)x, g(z)y)=\frac{1}{2\pi i}  K(x,y) \oint_{C_k} f dg, \,\,\, k=1,\dots, \np-1,
\eeq
where $C_k$ is a small contour surrounding $z=a_k.$

Let $\rg$ be a finite index subgroup of $\SLNZ[2]$ and $\mathfrak g$ be a simple Lie algebra, and assume that $\mfpk{1}$  contains an element that is never zero on $\uh$. In that case by our Theorem \ref{thm:alias odd group} we have an isomorphism $\alia\cong \mf{g}\otimes_{\mathbb C} \mfk{0}$
 and can define the cocycles
\beq{cocy3}
\omega_C(x\otimes f, y\otimes g)=\frac{1}{2\pi i}  K(x,y) \oint_C f dg, \quad f,g \in  \mfk{0}, \, x,y \in \mathfrak g
\eeq
for any contour $C \in H_1(Y(\Gamma), \mathbb Z)$. Bremner's results imply then that these cocycles generate all central extensions of the zero weight automorphic Lie algebra $\alia$ with smooth $X(\Gamma)$ of any genus.

In particular, this is true for  the principal congruence subgroups $\Gamma=\Gamma(N), \, N=3,4,5,6$ and the commutator subgroup of $\Gamma(1)$. The same is true for $\Gamma(2)$ with even representation, but for odd representations the answer depends on the number of invariant bilinear forms on 
$\mf{g}^{\rho(\pm\Id)}$ (see the proof of Theorem \ref{thm:alias even group}).

Note that the affine Kac-Moody algebras are the affine Lie algebras with added derivation in the usual case being $d=z\frac{d}{dz}$ (see \cite{kac1990infinite}). 
In general the choice of the derivation (vector field) is not obvious (see a relevant paper \cite{sheinman1990elliptic} by Sheinman, who studied this question for the Krichever-Novikov algebras in the genus 1 case).

Note that any derivation $\delta$ of a Lie algebra $\mathcal L$ defines the one-dimensional extension $\tilde{\mathcal L}=\mathcal L \oplus \mathbb C\delta,$ using the formula
$$
[x+\lambda \delta, y+\mu \delta]=[x,y]+\lambda \delta(y)-\mu \delta(x), \quad x,y \in \mathcal L.
$$
The following simple result provides an explicit example of such derivation in the modular case.

\begin{proposition}
For any finite index subgroup $\Gamma$ of modular group and any representation $\rho: \Gamma \to \Aut {\mathfrak g}$ the formula
\beq{der}
\delta = \frac{1}{2\pi i}\frac{E_4 E_6}{\Delta} \frac{\rd}{\rd\tau}
=(1-240 q- 141444 q^2 - 8529280 q^3 - 238758390 q^4-\ldots)\frac{\rd}{\rd q}
\eeq
defines a derivation of the corresponding $\alia$, which is the only weakly holomorphic derivation with the behaviour at the cusp $\delta=(1+o(q))\frac{d}{dq}$. On the absolute modular invariant $j$ it acts as
\beq{reld}
\delta (j)=-j(j-1728).
\eeq
\end{proposition}
In general, $a(\tau) \frac{\rd}{\rd\tau}$ with any weight $-2$ weakly holomorphic modular form $a(\tau)$ of $\Gamma$ defines a derivation of $\alia$, but $\delta$ has particularly nice properties. 
It appeared in \cite{gannon2014the} as an element of a family of derivation $\delta=\nabla_{1,0}$.
Formula \ref{reld} follows from Ramanujan's relations (\ref{ram1}).

Let us discuss now the representations of our automorphic Lie algebras. 

Consider the Lie algebra $\mathcal L(\mathfrak g, \Gamma)=\mf{g} \otimes_\mathbb C \mfk{0}$ and the following {\it evaluation representations} of $\mathcal L(\mathfrak g, \Gamma)$, extending the well-known construction for the loop algebra
$\mathcal L(\mathfrak g)=\mathfrak g\otimes \mathbb C[z, z^{-1}]$ (see e.g. \cite{chari1986integrable, chari1988integrable}).

Choose a tuple of points $a=(a_1, \dots, a_n), \, a_i \in Y(\Gamma)$ and a tuple $\psi=(\psi_1,\dots, \psi_n)$ of representations $\psi_i: \mathfrak g \to \End (V_i)$. 
The corresponding evaluation representation
$$ev_{a,\psi}: \mathcal L(\mathfrak g) \to \End(\bigotimes_{i=1}^n V_i)$$ 
is defined as
\beq{eval}
ev_{a,\psi}: f(z) \mapsto \sum_{i=1}^n \Id\otimes \dots \otimes \Id \otimes \psi_i(f(a_i))\otimes \Id \otimes \dots \otimes \Id.
\eeq
where $f(z) \in \mathcal L(\mathfrak g, \Gamma)$ is considered as a $\mathfrak g$-valued function on $Y(\Gamma).$

In the theory of the equivariant map algebras there are deep results by Lau \cite{lau2010representations} and Neher et al.~\cite{neher2012irreducible}, implying in particular that when $Y(\Gamma)$ is smooth (which means that $\Gamma$ has no elliptic elements) then all irreducible finite-dimensional representations of the Lie algebra $\mathcal L(\mathfrak g, \Gamma)$ are
evaluation representations (for the Onsager Lie algebra a similar result was proved earlier by Roan \cite{roan1991onsager}).

Combining this with our theorems \ref{thm:alias odd group}, \ref{thm:alias even group} we have in particular the following result.

\begin{Theorem}
\label{thm:reps}
Let $\mathfrak g$ be a complex perfect Lie algebra, $\rg\subset\SLNZ[2]$ be any finite index subgroup $\rg\subset\SLNZ[2]$ with smooth $Y(\Gamma)$ and $\rho:\rg\to\Aut{\mf{g}}$ be its restricted representation. Under assumptions of theorems \ref{thm:alias odd group}, \ref{thm:alias even group} all the irreducible representations of the automorphic Lie algebra $\alia$ have the form $ev_{a,\psi}\circ \Psi$, where $$
\Psi: \alia\cong \mf{g}\otimes_{\mathbb C} \mfk{0}.
$$ is the isomorphism (\ref{Psi0}),(\ref{Psi}).

In particular, this is true for the principal congruence subgroups $\Gamma(3k)$, $\Gamma(4k)$ and $\Gamma(5k)$, $k\in\mb{N}$ and any restricted representation $\rho.$
\end{Theorem}

In the theory of the equivariant map algebras \cite{lau2010representations}, \cite{neher2012irreducible} the group $\Gamma$ is assumed to be finite, so we cannot use their powerful results directly. To prove theorem \ref{thm:reps} we first use our isomorphism $\Psi$ and then apply the results of \cite{lau2010representations},\cite{neher2012irreducible} with trivial group acting on $Y(\Gamma).$

We can extend this using the results of Lau \cite{lau2014representations}, who introduced the notion of the twisted current algebras as a generalisation of the equivariant map algebras.
Let $\mf g$ be a simple complex Lie algebra and $A$ be a reduced finitely generated algebra over $\mathbb C$, $G$ be a finite group of automorphisms of the Lie algebra $\mf g\otimes_{\mathbb C} A,$ then the corresponding twisted current algebra can be defined as the fixed point subalgebra $(\mf g\otimes_{\mathbb C} A)^G.$ 

\begin{Theorem}
Let $\mf g$ be a semisimple Lie algebra, $\rg$ be a finite index subgroup of $\mathrm{SL}(2,\mb Z)$ and $\rho:\rg\to\Aut{\mf g}$ be any representation.
Assume that $\rg$ contains a finite index normal subgroup $N$ such that there exists a Lie algebra isomorphism $\aliavar{0}{\mf g}{N}{\rho}\cong \mf g \otimes_{\mb C} \mfk[N]{0}$ and consider $G=\rg/N.$ Then $\alia$ is isomorphic to the corresponding twisted current algebra
$(\mf g \otimes_{\mb C} \mfk[N]{0})^G$  and thus all its finite-dimensional irreducible representations are evaluation representations in the sense of Lau. 

In particular, this is true for all congruence subgroups $\rg \subset \mathrm{SL}(2,\mb Z)$ and restricted representations $\rho$.
\end{Theorem}
\begin{proof}The group $\rg$ acts on the weakly holomorphic maps $\uh\to\mf{g}$. Restricted to $\aliavar{0}{\mf g}{N}{\rho}$, this action descends to an action of $G=\rg/N$.
Averaging over this finite group shows that
$\aliavar{0}{\mf g}{N}{\rho}^{G}=\alia$. 
If we consider the action of $G$ on $\mf g \otimes_{\mb C} \mfk[N]{0}$ defined by the Lie algebra isomorphism $\aliavar{0}{\mf g}{N}{\rho}\cong \mf g \otimes_{\mb C} \mfk[N]{0}$ and the action of $G$ on $\aliavar{0}{\mf g}{N}{\rho}$ we obtain
\beq{isoL}
 \alia\cong (\mf g \otimes_{\mb C} \mfk[N]{0})^G.
 \eeq
Since $G$ is finite, we see that $\alia$ is isomorphic to a twisted current algebra, so we can apply the main result of Lau \cite{lau2014representations} to deduce that all finite-dimensional irreducible representations of $\alia$ are the evaluation representations, which Lau defined for all twisted current algebras.

The last claim now follows from Theorem 5.4 since any congruence subgroup by definition contains some principal congruence subgroup $\Gamma(k)$, and thus a normal subgroup $\Gamma(3k)$ for some $k \in \mathbb N.$
\end{proof}

In the case of automorphic Lie algebras $\aliapz$ it is natural to respect the grading. For the restricted representations we proved an isomorphism of
$\mb{Z}$-graded Lie algebras
$$
\modaut[]{}: \mf{g} \otimes_\mathbb C \mfpz \to \aliapz 
$$
where the grading of $\mf{g}$ in the left hand side is induced from $\rho.$

Let $\mu_i:\mfpz\to\End(W_i)$, $i=1,\ldots,n$, be $\mb{Z}$-graded representations of the commutative algebra $\mfpz$. Particular examples of such modules $W_i$ are quotients of $\mfpz$ by homogeneous ideals.
Let $\psi_i: \mathfrak g \to \End(V_i)$, $i=1,\ldots,n$, be $\mb{Z}$-graded Lie algebra representations.
Define the corresponding $\mb{Z}$-graded evaluation representation
\beq{eval2}
Ev_{\mu,\psi}=ev_{\mu,\psi} \circ \modaut[]{}^{-1}: \aliapz \to \End(\bigotimes_{i=1}^n V_i\otimes_\mathbb C W_i),
\eeq
where $ev_{\mu,\psi}: \mf{g} \otimes_\mathbb C \mfpz \to \End(\bigotimes_{i=1}^n V_i\otimes_\mathbb C W_i)$ is defined by
\beq{eval3}
ev_{\mu,\psi}: x \otimes \alpha \mapsto \sum_{i=1}^n \Id\otimes\ldots\otimes\Id\otimes \psi_i(x) \otimes \mu_i(\alpha)\otimes\Id\otimes\ldots\otimes\Id
\eeq
for $x\in \mathfrak g$ and $\alpha \in \mfpz$.

It would be interesting to study these representations and their role in the representation theory of automorphic Lie algebras in more detail.

\section{Concluding Remarks}
\label{sec:conclusions}

We have found a complete description of the automorphic Lie algebras on the upper half plane for representations $\rho:\rg\rightarrow\Aut{\mf{g}}$ factoring through $\SLNC[2]$.
The case of other representations $\rho$ is worthy to be  studied, and results are feasible due to the many strong papers from the last decade on vector-valued modular forms on the full modular group with an arbitrary representation. 

A particularly interesting class is given by Weil representations of $\mg$  representations from the theory of Jacobi modular forms, which can be naturally interpreted as the corresponding vector-valued modular forms  (see Eichler and Zagier \cite[Theorem 5.1]{eichler1985the}).

There is a well known Poisson structure on the algebra of modular forms that goes by the name of the first Rankin-Cohen bracket. This is an example of transvection, a tool in classical invariant theory to construct invariants, developed in the 19th century by Aronhold, Clebsch, Gordan and collaborators (see \cite{olver1999classical}). Using the complete sequence of transvectants one can construct a one-parameter family of associative products on the algebra of modular forms, known as the star- or Moyal product \cite[Proposition 5.20]{olver1999classical} (in fact, there exists a two-parameter family of deformations of the associative product, see \cite[p. 57]{bruinier2008the}).
The odd part of this product is known as the Moyal bracket. It is a deformation of the Poisson structure defined by the first transvectant. All of this structure can be generalised from scalar modular forms to vector-valued modular forms due to the fact that transvection can be done with vectors as well \cite{lee2007vector}. This will produce a one-parameter family of Lie brackets on the space of vector-valued modular forms $\vvmfpz$ if $V$ has the structure of an associative commutative algebra and $\rho:\rg\to\Aut{V}$. It would be interesting to explore this direction further in relation to Kontsevich's quantisation \cite{kontsevich2003deformation} (cf. \cite{pevzner2018a} and references therein).

We have also a link with the important work of Krichever and Novikov \cite{krichever1987algebras,krichever1987algebras2}, 
who studied the analogues of Virasoro algebras on Riemann surfaces $X$ with two punctures. The usual Virasoro algebra corresponds to the genus zero case. To define the extension of the Lie algebra of the vector fields in higher genus case, one needs a projective connection on $X$. In the modular case $X=\Gamma\backslash\uh$ we have a remarkable connection (going back to Ramanujan) given by the Serre derivatives
(\ref{ser}),
which work for any finite index subgroup $\Gamma$ of modular group. The corresponding multi-point generalisations of Krichever-Novikov algebras are worthy to be studied in more details (cf. \cite{schlichenmaier1990krichever,schlichenmaier2021krichever}).
  This may lead also to the new explicit solutions of the integrable nonlinear equation (like the famous KP equation), via original Krichever's construction \cite{krichever1984the},  
who was the first to use  the forms in algebro-geometric integration theory. 
Applications of automorphic Lie algebras on compact Riemann surfaces in the context of the theory of integrable systems have been considered in the past (e.g. in \cite{lombardo2004reductions}), and have been more recently discussed by Bury and Mikhailov in \cite{bury2021automorphic}.

\vspace{5mm}
{\bf Acknowledgement}
\vspace{1mm}
\\We are grateful to Fabien Cl\'{e}ry, Jenya Ferapontov, Anton Khoroshkin, Vladimir Novikov, Casper Oelen, Jan Sanders, Oleg Sheinman and Wadim Zudilin for very helpful and stimulating discussions.

The work of V.K. and S.L. was supported by the Engineering and Physical Sciences Research Council (EPSRC) grant no. EP/V019252/1. The work of A.P.V. was partially supported by the Russian Science Foundation grant no. 20-11-20214. 

\def\cprime{$'$}
\providecommand{\bysame}{\leavevmode\hbox to3em{\hrulefill}\thinspace}
\providecommand{\MR}{\relax\ifhmode\unskip\space\fi MR }
\providecommand{\MRhref}[2]{%
  \href{http://www.ams.org/mathscinet-getitem?mr=#1}{#2}
}
\providecommand{\href}[2]{#2}

\end{document}